\documentclass[journal]{IEEEtran}
\usepackage{amsmath,amssymb,amsfonts}
\usepackage{bm}
\usepackage{cite}
\usepackage{algorithm}
\usepackage{algorithmic}
\usepackage{graphicx}
\usepackage{caption}

\usepackage{url}
\usepackage{mathtools, cuted}
\usepackage{stfloats}
\usepackage{cleveref}
\usepackage{balance}
\usepackage{enumerate}
\usepackage{amsthm}
\usepackage{amsmath}
\usepackage{amssymb}
\usepackage{color}
\usepackage{booktabs}
\usepackage{lipsum}

\newcommand{\RNum}[1]{\uppercase\expandafter{\romannumeral #1\relax}}

\newtheorem{lemma}{Lemma}
\newtheorem{example}{Example}
\newtheorem{theorem}{Theorem}
\newtheorem{corollary}{Corollary}
\newtheorem{assumption}{Assumption}
\newtheorem{problem}{Problem}
\newtheorem{proposition}{Proposition}

\newtheorem{remark}{Remark}

\usepackage{bm}
\usepackage{threeparttable}
\usepackage{multirow} 
\usepackage{mathrsfs} 
 
\usepackage{mathtools}
\usepackage{makecell}

\newcommand\bigestimate{\makebox(0,0){\small$B_{\mK} V B_{\mK}^\tr$}}
\newcommand{\mK}{{\mathsf{K}}}
\newcommand{\mH}{{\mathsf{H}}}
\newcommand{\tr}{{{\mathsf T}}}

\ifCLASSOPTIONcompsoc
  \usepackage[caption=false,font=normalsize,labelfont=sf,textfont=sf]{subfig}
\else
  \usepackage[caption=false,font=footnotesize]{subfig}
\fi
\usepackage{tikz}
\newcommand\copyrighttext{%
  \footnotesize \textcopyright 2023 IEEE. Personal use of this material is permitted.
  Permission from IEEE must be obtained for all other uses, in any current or future
  media, including reprinting/republishing this material for advertising or promotional
  purposes, creating new collective works, for resale or redistribution to servers or
  lists, or reuse of any copyrighted component of this work in other works.}
\newcommand\copyrightnotice{%
\begin{tikzpicture}[remember picture,overlay]
\node[anchor=south,yshift=10pt] at (current page.south) {\fbox{\parbox{\dimexpr\textwidth-\fboxsep-\fboxrule\relax}{\copyrighttext}}};
\end{tikzpicture}%
}

\ifCLASSINFOpdf
\else
\fi
\begin{document}

%
\title{On the Optimization Landscape of Dynamic Output Feedback Linear Quadratic Control}
%
%
%

\author{Jingliang Duan, Wenhan Cao, Yang Zheng,  Lin Zhao
\thanks{The work of J. Duan was supported in part by the NSF China under Grant 52202487 and in part by the State Key Laboratory of Automotive Safety and Energy, China under Project KFY2212. The work of L. Zhao was supported by the Singapore Ministry of Education Tier 1 Academic Research Fund (A-0009030-00-00, 22-5460-A0001). Jingliang Duan worked on this project while he was a
postdoctoral fellow at the Department of Electrical and Computer Engineering, National University of Singapore. Corresponding author: L. Zhao}
\thanks{J. Duan is with the School of Mechanical Engineering, University of Science and Technology Beijing, China, and also with the Department of Electrical and Computer Engineering, National University of Singapore, Singapore. {\footnotesize Email:duanjl@ustb.edu.cn}.}

\thanks{W. Cao is with the School of Vehicle and Mobility, Tsinghua University, Beijing, 100084, China. {\footnotesize Email: cwh19@mails.tsinghua.edu.cn}.
}
\thanks{Y. Zheng is with the Department of Electrical and Computer Engineering, University of California San Diego, USA. {\footnotesize Email: zhengy@eng.ucsd.edu}.
}
\thanks{L. Zhao  is with the Department of Electrical and Computer Engineering, National University of Singapore, Singapore. {\footnotesize Email: elezhli@nus.edu.sg}.
}

}

 \maketitle
 \copyrightnotice


\begin{abstract}
The convergence of policy gradient algorithms hinges on the optimization landscape of the underlying optimal control problem. Theoretical insights into these algorithms can often be acquired from analyzing those of linear quadratic control. However, most of the existing literature only considers the optimization landscape for static full-state or output feedback policies (controllers). We investigate the more challenging case of dynamic output-feedback policies for linear quadratic regulation (abbreviated as \texttt{dLQR}), which is prevalent in practice but has a rather complicated optimization landscape. We first show how the \texttt{dLQR} cost varies with the coordinate transformation of the dynamic controller and then derive the optimal transformation for a given observable stabilizing controller. One of our core results is the uniqueness of the stationary point of \texttt{dLQR} when it is observable, which provides an optimality certificate for solving dynamic controllers using 
policy gradient methods. Moreover, we establish conditions under which \texttt{dLQR} and linear quadratic Gaussian control are equivalent, thus providing a unified viewpoint of optimal control of both deterministic and stochastic linear systems. These results further shed light on designing policy gradient algorithms for more general decision-making problems with partially observed information.
\end{abstract}


%
\IEEEpeerreviewmaketitle

\section{Introduction}
Reinforcement learning (RL) aims to directly learn optimal policies that minimize long-term cumulative costs through interacting with unknown environments. The past few years have witnessed great successes of RL in various domains, such as video games \cite{mnih2015human,vinyals2019grandmaster}, robots control \cite{nguyen2019review_RLrobot}, nuclear fusion \cite{degrave2022magnetic}, and recommender systems \cite{zou2019recommender}. Despite the impressive empirical performance of many policy gradient algorithms (such as DDPG \cite{lillicrap2015DDPG}, PPO \cite{schulman2017PPO}, SAC \cite{Haarnoja2018SAC}, DSAC \cite{duan2021distributional}),  theoretical guarantees of their convergence, optimality, and sample complexity remain under-explored and a big challenge.

As a case study, canonical optimal control of linear time-invariant (LTI) systems has been commonly analyzed to help reveal various theoretical properties of policy gradient methods \cite{fazel2018global,bu2019lqr,mohammadi2019CT-LQR,malik2019derivative,hu2022towards}. 
In particular, the linear quadratic regulator (LQR) has regained significant research interest \cite{fazel2018global,bu2019lqr,mohammadi2019CT-LQR,malik2019derivative}. It is well known that the optimal LQR controller is a static linear state feedback policy, and the set of all stabilizing state-feedback gains is path-connected for both discrete-time and continuous-time LTI systems. Furthermore, recent investigations from the learning perspective show that the LQR cost function enjoys an interesting property of gradient dominance~\cite{fazel2018global,mohammadi2019CT-LQR}. This enables a global linear convergence guarantee for a variety of gradient descent methods \cite{fazel2018global,yang2019global}, despite the non-convexity of LQR. An increasing body of subsequent studies has sought to delineate the properties of policy gradient methods in application to different control 
problems for LTI systems, including finite-horizon noisy LQR \cite{hambly2020noisyLQR}, LQR tracking \cite{ren2021lqrtracking},  Markovian jump LQR \cite{jansch2020Mjump}, linear $\mathcal{H}_2$ control with $\mathcal{H}_{\infty}$ constraints \cite{zhang2020robust}, finite MDPs \cite{bhandari2019global}, and risk-constrained LQR \cite{zhao2021primal}.

\vspace{-1.5pt}

The aforementioned literature mainly focuses on the case of static full state-feedback control. In many practical settings, the complete state information of the underlying system may not be directly available. Some recent works have studied static output-feedback (SOF) controllers to optimize a linear quadratic cost function~\cite{duan2021optimization,fatkhullin2021CTSOF,feng2020connectivity,bu2019topological}. Different from the full state-feedback LQR, it is shown that policy gradient methods for solving optimal SOF controllers do not possess the gradient dominance property and thus are unlikely to find the globally optimal solution.  
In fact, the set of stabilizing SOF controllers is typically disconnected, and the stationary points can be local minima, saddle points, or even local maxima \cite{fatkhullin2021CTSOF,feng2020connectivity}. Moreover, even finding a stabilizing SOF controller is a challenging task~\cite{blondel1997np,syrmos1997static}. In addition to the SOF controller, the global convergence for a class of distributed finite-horizon output-feedback LQR problems was established in \cite{furieri2020distributedLQR}, where the policy is subject to subspace constraints and represented by a linear combination of all historical outputs. However, the result does not generalize to infinite-horizon optimal control problems.

This paper takes a step further to analyze the optimization landscape of the infinite-horizon dynamic output-feedback LQR (\texttt{dLQR}). From linear control theory, a \textit{stabilizing} dynamic controller for~\texttt{dLQR} can be found via designing separately a stable observer and a state-feedback controller thanks to the separation principle~\cite{lewis2012optimal}. In the context of RL, an observer-based \textit{optimal} dynamic controller may be learned through gradient descent optimization of the LQR cost. Since (policy) gradient descent is the main workhorse for deep RL, it is of great interest to study its capability via the lens of canonical linear optimal control problems. The recent work~\cite{mohammadi2021lack} showed the non-uniqueness of stationary points of learning an observer-based dynamic controller for the classical Linear Quadratic Gaussian (LQG) control problem, where the controller requires complete knowledge of the system model. The closely related work \cite{zheng2021analysis} instead considered a \textit{general} full-order dynamic controller without the parameterization using system matrices. It was found that all stationary points that correspond to minimal controllers (i.e., whose state-space realization is reachable and observable) are globally optimal to LQG, and that these stationary points are identical up to coordinate (similarity) transformations \cite{zheng2021analysis}. Different from LQG which considers stochastic linear systems and minimizes a limiting \textit{average} cost (or the variance of the steady state), \texttt{dLQR} seeks a dynamic controller that minimizes an infinite-horizon \textit{accumulated} cost for a deterministic LTI system. In the latter case, the cost is influenced by both the similarity transformation and the system transient dynamics induced by the initial system and controller states, which suggests a more complicated optimization landscape. Indeed, little is known about the optimality of the converged solutions of policy gradient methods for solving \texttt{dLQR}. The analysis of LQG in~\cite{zheng2021analysis,mohammadi2021lack,zheng2022escaping} does not extend to the \texttt{dLQR} directly.

In this paper, we provide a comprehensive analysis of the stationary points of \texttt{dLQR}, characterize its optimality in a general setting, and in particular, establish conditions under which \texttt{dLQR} and LQG are equivalent. Specifically,
\begin{enumerate}
    \item We analyze the impact of similarity transformations on the \texttt{dLQR} cost and derive an explicit form of the unique optimal similarity transformation for a given observable (i.e., whose state-space realization is observable) stabilizing controller. 
    \item We show the observable stationary point of \texttt{dLQR} is unique and in a concise form of an observer-based controller with the optimal similarity transformation. Despite the non-connectivity of the stabilizing domain, this result enables us to characterize a set of conditions under which the global optimality can be achieved if a policy gradient method for solving \texttt{dLQR} converges. 
    \item Finally, we establish an interesting equivalent correspondence to LQG, and further prove that under a certain structural constraint on the initial state of the dynamic controller, \texttt{dLQR} enjoys the same symmetry properties induced by similarity transformations and the global
    optimality of minimal stationary points. It provides a unified viewpoint of optimal control of deterministic and stochastic LTI systems. 
\end{enumerate}

Throughout the paper, we provide a few numerical examples to illustrate our theoretical analysis. These findings bring new insights into the performance of policy gradient methods for solving deterministic and stochastic optimal control problems with partially observed states. 
The remainder of this paper is organized as follows.  Section \ref{sec:preliminary} presents the problem statement of \texttt{dLQR}, and Section \ref{sec.formulation} derives an analytical form of the \texttt{dLQR} cost as a function of dynamic controller parameters. Section \ref{sec.similarity_transformation} analyzes the impact of similarity transformations on the \texttt{dLQR} cost. We derive the formula of the gradient and characterize the structure of the observable stationary controller in Section \ref{sec.mainresults}. Section \ref{sec:lqr_lqg} analyzes the relationship between \texttt{dLQR} and LQG. We present numerical experiments in Section \ref{sec:numerical_experi} and provide conclusions in Section \ref{sec:conclusion}. Some technical proofs are provided in the appendix.

\textbf{Notation:} 
Given a matrix $X \in \mathbb{R}^{n \times n}$, we use $\rho(X)$, ${\rm Tr}(X)$, $\lambda_{\rm min}(X)$, and $\|X\|_F$ to denote its spectral radius, trace, minimum eigenvalue (for symmetric matrices), and Frobenius norm, respectively. The notation $\mathbb{S}^n_{+}$ (respectively, $\mathbb{S}^n_{++}$) denotes the set of symmetric $n \times n$ positive semidefinite (respectively, positive definite) matrices. We use $X \succ Y$ ($X \succeq Y$) to represent that $X-Y$ is positive definite (semidefinite). Finally, $\mathrm{GL}_n$ denotes the set of $n\times n$ invertible matrices, and $I_n$ denotes the identity matrix. 

\section{Problem Statement}
\label{sec:preliminary}
In this section, we briefly review the canonical linear quadratic optimal control problem, 
and then motivate the dynamic output-feedback Linear Quadratic Regulator (\texttt{dLQR}). 

\subsection{Linear Quadratic Control}
Consider a discrete-time LTI system 
\begin{equation}  
\label{eq.statefunction}
\begin{aligned}
x_{t+1} &= Ax_t+Bu_t,\\
y_t &= Cx_t,
\end{aligned}
\end{equation}
where $A\in \mathbb{R}^{n\times n}$, $B\in \mathbb{R}^{n\times m}$, $C\in \mathbb{R}^{d\times n}$ are system~matrices, and $x_t\in \mathbb{R}^n$, $u_t\in \mathbb{R}^m$,  $y_t\in \mathbb{R}^d$ are the system state, input, and output measurements at time $t$, respectively. The linear quadratic control seeks a sequence $u_0, u_1, \ldots, u_t, \ldots $ minimizing the infinite-horizon accumulated cost:
\begin{equation}   
\label{eq.objective}
\begin{aligned}
\min_{u_t} \;\; &\mathbb{E}_{x_0\sim \mathcal{D}}\left[\sum_{t=0}^{\infty}\left(x_t^\tr Qx_t+u_t^\tr Ru_t\right) \right] \\
\text{subject to}\;\;&~\eqref{eq.statefunction},
\end{aligned}
\end{equation}
where $Q \in \mathbb{S}_+^{n}$ and $R \in \mathbb{S}_{++}^{m}$ are performance weights, and the control input $u_t$ at time $t$ is allowed to depend on the historical outputs $y_{0}, y_1, \ldots,y_t$ and inputs $u_{0}, u_1, \ldots, u_{t-1}$. We assume that $\mathbb{E}_{x_0\sim \mathcal{D}}[x_0x_0^\tr] \succ 0$, where $\mathcal{D}$ represents the initial state distribution. This assumption is quite standard for learning-based control \cite{fazel2018global,  lee2018primal,malik2019derivative,hu2022towards} and might be deemed as the persistent excitation condition in the data-driven control literature \cite{de2019formulas}. For problem \eqref{eq.objective}, we make the following standard assumption: 
\begin{assumption}
\label{assumption.control_observe}
$(A,B)$ is controllable, and $(C, A)$ and $(Q^{\frac{1}{2}},A)$ are observable.
\end{assumption}

Without loss of generality, we assume $C$ has full row rank.  The state-feedback LQR corresponds to $C=I_n$. In this special case, the globally optimal controller is a static linear feedback $u_t = K x_t$, where $K \in \mathbb{R}^{m \times n}$ can be obtained via solving a Riccati equation~\cite{bertsekas2017dynamic}. In general cases where $\text{rank}(C)<n$, a static output-feedback (SOF) gain $u_t = Ky_t$ with $K \in \mathbb{R}^{m \times d}$ is typically insufficient to obtain good control performance. In fact, the set of stabilizing SOF gains can be highly disconnected \cite{feng2020connectivity}, and even finding a stabilizing SOF controller is generally a challenging task~\cite{blondel1997np,syrmos1997static}. Unlike SOF control, under Assumption \ref{assumption.control_observe}, a stabilizing \textit{dynamic} output-feedback controller always exists and can be found easily, thanks to the well-known separation principle \cite{lewis2012optimal}. 
In particular, the following standard observer-based controller can be designed to stabilize the plant~\eqref{eq.statefunction}
\begin{equation} \label{eq:observer-based-controller}
\begin{aligned}
    \xi_{t+1} &= (A - BK-LC) \xi_t + Ly_{t}, \\
    u_t &= -K \xi_{t},
\end{aligned}
\end{equation}
where $K \in \mathbb{R}^{m \times n}, L \in \mathbb{R}^{n \times d}$ are the feedback gain and observer gain matrices such that $A - BK$ and $A - LC$ are stable~\cite{zhang2020robust}, and $\xi_t \in \mathbb{R}^n$ is the internal state of the controller.

\subsection{The \texttt{dLQR} Problem}
\label{sec.setting}
In this paper, we assume that the  order of the system state $n$ is known. Motivated by learning deterministic dynamic controllers, we consider the class of \textit{full-order} dynamic output-feedback controllers in the form of \footnote{This is in the standard form of strictly proper dynamic controllers, where there is no direct feed-through of $y_t$ to $u_t$ \cite{zheng2021analysis,van2020data}.}
\begin{equation}   
\label{eq.dynamic_controller}
\begin{aligned}
\xi_{t+1} &= A_{\mK}\xi_t + B_{\mK}y_t,\\
u_t &= C_{\mK}\xi_t,
\end{aligned}
\end{equation}
where matrices $C_{\mK}\in \mathbb{R}^{m\times n}$, $B_{\mK} \in \mathbb{R}^{n\times d}$, $A_{\mK} \in \mathbb{R}^{n\times n}$ are the controller parameters to be learned. It is clear that the observer-based controller \eqref{eq:observer-based-controller} is a special case of \eqref{eq.dynamic_controller}. We also note that the controller parameterization in \eqref{eq.dynamic_controller} does not explicitly rely on the knowledge of system parameters $A$, $B$, and $C$, which allows for model-free policy learning.

Let $\xi_0$ be the initial controller state and suppose $(x_0,\xi_0)$ follows a joint distribution $\bar{\mathcal{D}}$. In the case of output-feedback LQR problem,  different initial state distributions often result in distinct optimal output-feedback controllers \cite[Proposition 1]{duan2021optimization}.  In addition to $A_{\mK}$, $B_{\mK}$, and $C_{\mK}$, the transient behavior induced by the initial controller state (or exchangeably, initial state estimate) also affects the accumulated cost. Hence, the optimality of the dynamic output-feedback controller is contingent upon a specific initial joint distribution $\bar{\mathcal{D}}$. Consequently, it is necessary to incorporate the expectation over $\bar{\mathcal{D}}$ when formulating the optimization objective. Then, the dynamic output-feedback LQR (\texttt{dLQR}) which aims to minimize the accumulated linear quadratic cost \cite{modares2016optimal,rizvi2018output,rizvi2019reinforcement,rizvi2020output} is given by
\begin{equation}   
\label{eq.dLQR}
\begin{aligned}
\min_{A_{\mK},B_{\mK},C_{\mK}} \;\; &\mathbb{E}_{(x_0,\xi_0)\sim \bar{\mathcal{D}}}\left[\sum_{t=0}^{\infty}\left(x_t^\tr Qx_t+u_t^\tr Ru_t\right) \right] \\
\text{subject to}\;\;&~\eqref{eq.statefunction},~\eqref{eq.dynamic_controller}.
\end{aligned}
\end{equation}

Problem \eqref{eq.dLQR} presents a general formulation for learning dynamic output-feedback controllers, the analysis of which includes several existing results as special cases under a unified viewpoint. Typically, the distribution of the initial system state can be determined by the practical optimal control tasks, whereas the initial controller state $\xi_0$ can either be chosen from a pre-specified distribution ($\xi_0$ is independent of $\mK$) or treated as a variable to be learned (e.g., $\xi_0$ is a function of $\mK$). Problem \ref{pro.dLQR} in Section \ref{subsection: problem-1} considers the case where $\xi_0$ is independent of $\mK$; Problem \ref{pro.dLQR_initial_estimate} in Section \ref{subsetion:problem-2} considers that case where $\xi_0$ is a function of $\mK$. As a preview of our results, distinct optimization landscapes summarized in \Cref{table.summary} can be delineated for Problem~\eqref{eq.dLQR}, conditioned on different assumptions on the initial distribution $\bar{\mathcal{D}}$.  


\begin{table*}[!htb]
\centering
\caption{Optimization Landscapes of \texttt{dLQR} under Different Cases}
\label{table.summary}
\begin{tabular}{cccc}
\toprule
Assumption on $\xi_0$&\makecell{Cross-correlation of \\$\xi_0$ and $x_0$} & \makecell{Invariance under \\similarity transformation } &  Properties of stationary points  \\
\midrule
\multirow{2}{*}{\makecell{$\xi_0$ is independent of $\mK$\\(see Problem \ref{pro.dLQR})}}& nonsingular& \multirow{2}{*}{\makecell{no \\(Propositions \ref{prop.similarity_variance} and \ref{proposition.optimal_tansformation})}} &\makecell[l]{ uniqueness and the explicit structure of the observable stationary\\ point (\Cref{theorem.solution_expression})}\\
\cmidrule{2-2}\cmidrule{4-4}
& singular&  &\makecell[l]{non-existence of the observable stationary point (\Cref{corollary.exist_of_solution})}\\
\midrule
\makecell{$\xi_0$ is a linear function of $B_{\mK}$\\ (see Problem \ref{pro.dLQR_initial_estimate})}&zero matrix& \makecell{yes\\ (\Cref{theorem.optimal_equivalence})}& \makecell[l]{non-uniqueness, global optimality, and the explicit structure  of\\ minimal stationary points (\Cref{theorem.optimal_equivalence})}\\
\bottomrule
\end{tabular}
\end{table*}

\begin{remark}
In the classical LQG control,  there are additive white Gaussian process and measurement noises in the LTI system~\eqref{eq.statefunction}, and the LQG objective focuses on minimizing an average cost, i.e., the final state covariance. Consequently, the transient behavior is not important in the classical LQG problem. On the contrary, our \texttt{dLQR}~\eqref{eq.dLQR}  aims to  minimize an infinite-horizon accumulated cost, in which the system transient behavior induced by the initial system and controller states is considered. Therefore, the recent results on the landscape analysis of LQG control in \cite{zheng2021analysis,zheng2022escaping} are not directly applicable to the \texttt{dLQR} problem. We will further clarify the connections and differences between the standard LQG and our \texttt{dLQR} in Section \ref{sec:lqr_lqg}. 
\end{remark}

\section{Optimization formulation of the \texttt{dLQR} Problem}
\label{sec.formulation}

In this section, we derive the analytical form of the cost function \eqref{eq.dLQR} in terms of the dynamic controller parameters, which is needed for analyzing its optimization landscape.  
%
\subsection{Cost function in the \texttt{dLQR} Problem} \label{subsection: problem-1}
The closed-loop system of \eqref{eq.statefunction} under \eqref{eq.dynamic_controller} is 
\begin{equation}   
\label{eq.closed-loop-system}
\begin{bmatrix}
     x_{t+1}\\
     \xi_{t+1} 
\end{bmatrix} =  \begin{bmatrix}
     A & BC_{\mK}\\
     B_{\mK}C & A_{\mK}
\end{bmatrix}\begin{bmatrix}
     x_{t}\\
     \xi_{t} 
\end{bmatrix}.
\end{equation}
We further denote 
$$
    \bar{x}_t := \begin{bmatrix}
x_t \\
\xi_t
\end{bmatrix},  \bar{A} := \!\begin{bmatrix}
     A & 0\\
     0 & 0
\end{bmatrix},  \bar{B} :=  \!\begin{bmatrix}
     B & 0 \\
     0 & I_n
\end{bmatrix}, \bar{C}:= \!\begin{bmatrix}
     C & 0\\
     0 & I_n
\end{bmatrix},
$$
and write the controller parameters in a compact form
\begin{equation}
\label{eq.K_structure}
\mK := \begin{bmatrix}
0_{m\times d} & C_{\mK} \\
B_{\mK} & A_{\mK}
\end{bmatrix}.   
\end{equation}

Then \eqref{eq.closed-loop-system} can be expressed as
\begin{equation}
\label{eq.closed-loop-system_short}
\bar{x}_{t+1} = (\bar{A}+\bar{B} \mK \bar{C})\bar{x}_{t}.
\end{equation}
The set $\mathbb{K}$ of all stabilizing controllers is given by
\begin{equation}
\label{eq:stabilizing-K}
\mathbb{K}:=\left\{\mK = \begin{bmatrix}
0_{m\times d} & C_{\mK} \\
B_{\mK} & A_{\mK}
\end{bmatrix}:\rho(\bar{A}+\bar{B}\mK\bar{C})<1\right\}.
\end{equation}
It is known that $\mathbb{K}$ is non-convex but has at most two disconnected components~\cite{zheng2021analysis}. Upon denoting 
$$\bar{Q}=\begin{bmatrix}
     Q & 0_{n\times n}\\
     0_{n\times n} & 0_{n\times n}
\end{bmatrix}, \; \bar{R}=\begin{bmatrix}
     R & 0_{m\times n}\\
     0_{n\times m}& 0_{n\times n}
\end{bmatrix},$$ 
the \texttt{dLQR} problem \eqref{eq.dLQR} can be written as
\begin{equation}  
\label{eq.objective_with_K}
\begin{aligned}
\min_{\mK} \quad&  \mathop{\mathbb{E}}_{\bar{x}_0\sim \bar{\mathcal{D}}}\left[\sum_{t=0}^{\infty}\bar{x}_t^\tr \left(\bar{Q}+\bar{C}^\tr \mK^\tr \bar{R}\mK\bar{C} \right)\bar{x}_t \right]\\
\text{subject to} \quad & \eqref{eq.closed-loop-system_short}, \; \mK\in \mathbb{K}.
\end{aligned}
\end{equation}

For the LTI system \eqref{eq.closed-loop-system_short}, the value function of state $\bar{x}_t$ under a stabilizing controller $\mK \in \mathbb{K}$ takes a quadratic form~as
\begin{equation}
\nonumber
V_{\mK}(\bar{x}_t): = \bar{x}_t^\tr P_{\mK}\bar{x}_t,
\end{equation}
where $P_{\mK} \in \mathbb{S}_{+}^{2n}$ satisfies a Lyapunov equation (see \Cref{lemma.dLQR_cost} below). Define the accumulated state correlation matrix under a stabilizing controller $\mK \in \mathbb{K}$ as
\begin{equation}
\nonumber
\Sigma_{\mK}:=\mathbb{E}_{\bar{x}_0\sim \bar{\mathcal{D}}}\sum_{t=0}^{\infty}\bar{x}_t \bar{x}_t^\tr.
\end{equation}
With each $\mK \in \mathbb{K}$, we can define a $P_{\mK}$ and a $\Sigma_{\mK}$ accordingly. These two matrices can be used to express the \texttt{dLQR} cost value under $\mK$ conveniently.  This is summarized below.  
\begin{lemma}
\label{lemma.dLQR_cost}
For any $ \mK \in \mathbb{K}$, the \texttt{dLQR} cost value is given by
\begin{equation}   
\label{eq.cost_in_P}
J(\mK) = {\rm Tr}(P_{\mK}X)= {\rm Tr}\left(\begin{bmatrix}
   Q  & 0 \\
   0  &  C_{\mK}^\tr RC_{\mK}
\end{bmatrix}\Sigma_{\mK}\right),
\end{equation}
where $P_{\mK}$ and $\Sigma_{\mK}$ are the unique positive semidefinite solutions to the following Lyapunov equations
\begin{subequations}
\begin{align} 
\label{eq.lyapunov_equation}
P_{\mK} &= \bar{Q} + \bar{C}^\tr \mK^\tr \bar{R}\mK\bar{C} \\
&\qquad \qquad +(\bar{A}+\bar{B}\mK\bar{C})^\tr P_{\mK}(\bar{A}+\bar{B}\mK\bar{C}),  \nonumber  \\
\Sigma_{\mK} &= X +(\bar{A}+\bar{B}\mK\bar{C})\Sigma_{\mK}(\bar{A}+\bar{B}\mK\bar{C})^\tr, \label{eq.lyapunov_equation_sigma}
\end{align} 
\end{subequations}
and $X :=  \mathbb{E}_{\bar{x}_0\sim \bar{\mathcal{D}}} \; [\bar{x}_0\bar{x}_0^\tr]$ denotes the initial state correlation matrix. 
\end{lemma}
According to Lemma \ref{lemma.dLQR_cost}, the objective $J(\mK)$ is impacted by the initial distribution $\bar{\mathcal{D}}$ through the initial correlation matrix $X$. For theoretical analysis of the landscape, only the correlation matrix is needed, and there is no requirement for a specific type of the initial distribution $\bar{\mathcal{D}}$. Finally, we formulate the \texttt{dLQR} problem \eqref{eq.dLQR} into the following static optimization form.

\begin{problem}[Policy optimization for \texttt{dLQR} where $\xi_0$ is independent of $\mK$]
\label{pro.dLQR}
\begin{equation}  
\nonumber
\begin{aligned}
\min_{\mK} \quad&  J(\mK)\\
\text{\rm subject to} \quad &\mK\in \mathbb{K},
\end{aligned}
\end{equation}
where $J(\mK)$ is defined in~\eqref{eq.cost_in_P} and $\mathbb{K}$ is given in~\eqref{eq:stabilizing-K}. The initial controller state is assumed to follow a fixed distribution, and thus the initial state correlation matrix $X$ in~\eqref{eq.cost_in_P} is independent of parameters $\mK$. 
\end{problem}

We will derive the analytical policy gradients of $J(\mK)$ to analyze the optimization landscape of Problem 1. 


\subsection{Block-wise Lyapunov equations and useful lemmas}
The block-wise Lyapunov equations in \eqref{eq.lyapunov_equation} and \eqref{eq.lyapunov_equation_sigma} will be used extensively in this paper. We write  
\begin{equation} \label{eq:Pk-partition}
P_{\mK}=\begin{bmatrix}
   P_{\mK,11}  &  P_{\mK,12}\\
   P_{\mK,12}^\tr  & P_{\mK,22} 
\end{bmatrix}. 
\end{equation}
In the sequel, the subscript $\mK$ of submatrices of $P_{\mK}$ and $\Sigma_{\mK}$  will be omitted when the dependence on $\mK$ is clear from the context. From \eqref{eq.lyapunov_equation}, we have
\begin{subequations}
\begin{align}
&\begin{aligned}
\label{eq.block_lyapunov_P11}
P_{11}&=Q+ A^\tr P_{11} A +C^\tr B_{\mK}^\tr P_{12}^\tr A \\
&\qquad\qquad + A^\tr P_{12} B_{\mK} C+C^\tr B_{\mK}^\tr P_{22} B_{\mK} C,
\end{aligned}\\
&\begin{aligned}
\label{eq.block_lyapunov_P12}
P_{12}&=A^\tr P_{11}BC_{\mK} + C^\tr B_{\mK}^\tr P_{12}^\tr B C_{\mK} \\
&\qquad\qquad + A^\tr P_{12} A_{\mK} +C^\tr B_{\mK}^\tr P_{22} A_{\mK},
\end{aligned}\\
&\begin{aligned}
\label{eq.block_lyapunov_P22}
P_{22}&=C_{\mK}^\tr RC_{\mK}+ A_{\mK}^\tr P_{12}^\tr B C_{\mK} + C_{\mK}^\tr B^\tr P_{12} A_{\mK} \\
&\qquad\qquad +C_{\mK}^\tr B^\tr P_{11} B C_{\mK}+A_{\mK}^\tr P_{22}A_{\mK}.
\end{aligned}
\end{align}
\end{subequations}
Similarly, upon letting
\begin{equation} \label{eq:Sigma-X-partition}
\Sigma_{\mK}=\begin{bmatrix}
   \Sigma_{\mK,11}  &  \Sigma_{\mK,12}\\
   \Sigma_{\mK,12}^\tr  & \Sigma_{\mK,22} 
\end{bmatrix}, \;\; X=\begin{bmatrix}
   X_{11}  &  X_{12}\\
  X_{12}^\tr  & X_{22} 
\end{bmatrix},
\end{equation} 
we get
\begin{subequations}
\label{eq.block_lyapunov_sigma}
\begin{align}
&\begin{aligned}
\label{eq.block_lyapunov_sigma11}
\Sigma_{11}&=X_{11}+A \Sigma_{11} A^\tr +BC_{\mK} \Sigma_{12}^\tr A^\tr\\
&\qquad \qquad + A \Sigma_{12} C_{\mK}^\tr B^\tr+B C_{\mK} \Sigma_{22} C_{\mK}^\tr B^\tr,
\end{aligned}\\
&\begin{aligned}
\label{eq.block_lyapunov_sigma12}
\Sigma_{12}&=X_{12}+  A \Sigma_{11}C^\tr B_{\mK}^\tr + B C_{\mK} \Sigma_{12}^\tr C^\tr B_{\mK}^\tr \\
&\qquad \qquad + A \Sigma_{12} A_{\mK}^\tr +B C_{\mK} \Sigma_{22} A_{\mK}^\tr,
\end{aligned}\\
&\begin{aligned}
\label{eq.block_lyapunov_sigma22}
\Sigma_{22}&=X_{22}+B_{\mK} C \Sigma_{11} C^\tr B_{\mK}^\tr + A_{\mK}\Sigma_{12}^\tr C^\tr B_{\mK}^\tr \\
&\qquad \qquad + B_{\mK}C \Sigma_{12} A_{\mK}^\tr+A_{\mK} \Sigma_{22}A_{\mK}^\tr.
\end{aligned}
\end{align}
\end{subequations}

Standard Lyapunov theorems will be used throughout the paper. We summarize them below for completeness.
\begin{lemma}[Lyapunov stability theorems \cite{gu2012discrete,lee2018primal}]
\label{lemma.Lyapunov_stability}
\ 
\begin{enumerate}[(a)]
\item If $\rho(A) <1$ and $Q \in \mathbb{S}_{+}^n$, the Lyapunov equation $P=Q+A^\tr P A$ has a unique solution $P \in \mathbb{S}_{+}^n$.
\item Let $Q \in \mathbb{S}_{++}^n$. $\rho(A) <1$ if and only if there exists a unique $P \in \mathbb{S}_{++}^n$ such that $P=Q+A^\tr P A$. 
\item Suppose $(D,A)$ is observable (or $(A^\tr,D^\tr)$ is reachable).  $\rho(A) <1$ if and only if there exists a unique $P \in \mathbb{S}_{++}^n$ such that $P=D^\tr D+A^\tr P A$. 
\end{enumerate}
\end{lemma}

\section{\texttt{dLQR} Cost under Different Similarity Transformations}
\label{sec.similarity_transformation}

For dynamic controllers, a widely used concept is the so-called \textit{similarity transformation}~\cite{zhou1996robust}. It can be shown that similarity transformations do not change the control performance of the LQG problem~\cite[Lemma 4.1]{zheng2021analysis}. In contrast, the \texttt{dLQR} cost varies with different similarity transformations due to the transient behavior induced by initial controller states, and thus the optimization landscape of \texttt{dLQR} is distinct.

\subsection{Varying \texttt{dLQR} cost}

Given a controller $\mK$ and an invertible matrix $T\in \mathrm{GL}_n$, we define the similarity transformation of $\mK$ as
\begingroup
\def\arraystretch{0.95}
\setlength\arraycolsep{2.5pt}
\begin{equation} \label{eq:similarity-transformation}
\begin{aligned}
\mathscr{T}_T(\mK) \!= \!\begin{bmatrix}
I_m & 0 \\
0 & T
\end{bmatrix}\mK\begin{bmatrix}
I_d & 0 \\
0 & T
\end{bmatrix}^{-1}\!\!
=\!\begin{bmatrix}
0 & C_{\mK}T^{-1} \\
TB_{\mK} & TA_{\mK}T^{-1}
\end{bmatrix}.
\end{aligned}
\end{equation}
\endgroup
It is easy to verify that if $\mK\in \mathbb{K}$ and $T\in \mathrm{GL}_n$, then $\mathscr{T}_T(\mK) \in \mathbb{K}$ \cite{zheng2021analysis}. 
The following result provides the formula to calculate the cost after a similarity transformation~\eqref{eq:similarity-transformation}.
\begin{proposition}
\label{prop.similarity_variance}
Let $\mK\in \mathbb{K}$ and $T\in \mathrm{GL}_n$. We have
\begin{equation}
\label{eq.cost_transformation}
J(\mathscr{T}_T(\mK))  = {\rm Tr}\left(P_{\mK}\bar{T}^{-1}X\bar{T}^{-\tr}\right),
\end{equation}
where $\bar{T}:=\begin{bmatrix}
I_n & 0 \\
0 & T
\end{bmatrix}$, $P_{\mK}$ is the unique positive semidefinite solution to~\eqref{eq.lyapunov_equation}, and $X =  \mathbb{E}_{\bar{x}_0\sim \bar{\mathcal{D}}} \; [\bar{x}_0\bar{x}_0^\tr]$. 
\end{proposition}

\begin{proof}
Since $\mK\in \mathbb{K}$, by Lemma \ref{lemma.Lyapunov_stability}(a), the Lyapunov equation \eqref{eq.lyapunov_equation} admits a unique positive semidefinite solution for both $\mK$ and $\mathscr{T}_T(\mK)$. Hence, the solution of \eqref{eq.lyapunov_equation} given $\mK$ can be expressed as 
\begin{equation*}
P_{\mK}= \sum_{k=0}^{\infty}\left((\bar{A}+\bar{B}\mK\bar{C})^\tr\right)^k \begin{bmatrix}
   Q  & 0 \\
   0  &  C_{\mK}^\tr RC_{\mK}
\end{bmatrix}(\bar{A}+\bar{B}\mK\bar{C})^k.  
\end{equation*}
Similarly, by the definition of $\mathscr{T}_T(\mK)$ in~\eqref{eq:similarity-transformation}, one has
\begin{equation}
\label{eq.similarity_P}
P_{\mathscr{T}_T(\mK)}= \bar{T}^{-\tr}P_{\mK}\bar{T}^{-1}.
\end{equation}
Therefore, by \eqref{eq.cost_in_P}, we have
\begin{equation}
\nonumber
J(\mathscr{T}_T(\mK)) = {\rm Tr}\left(P_{\mathscr{T}_T(\mK)}X\right)  = {\rm Tr}\left(P_{\mK}\bar{T}^{-1}X\bar{T}^{-\tr}\right),
\end{equation}
which completes the proof.
\end{proof}

Proposition \ref{prop.similarity_variance} characterizes how the \texttt{dLQR} cost varies with similarity transformations. The initial controller state~$\xi_0$~is~assumed to follow a fixed distribution and the similarity transformation implies a coordinate change of the internal controller state. If the controller coordinate changes while its initial~state does not change, this leads to a different controller~\eqref{eq.dynamic_controller}, which naturally results in a different \texttt{dLQR} cost value. 

\subsection{Optimal similarity transformation}
One natural consequence of Proposition \ref{prop.similarity_variance} is that for each stabilizing controller $\mK\in \mathbb{K}$, there might exist an optimal similarity transformation matrix $T^\star$  
in the sense that 
\begin{equation}
\label{eq.optimal_T}
J(\mathscr{T}_{T^\star}(\mK)) \leq J(\mathscr{T}_T(\mK)), \quad \forall T\in \mathrm{GL}_n. 
\end{equation}

We refer to \eqref{eq.dynamic_controller} as an \textit{observable  controller} if $(C_{\mK},A_{\mK})$ is observable. We denote the set of observable controllers as 
\begin{equation}
\nonumber
\mathbb{K}_o := \left\{ \begin{bmatrix}
0_{m\times d} & C_{\mK} \\
B_{\mK} & A_{\mK}
\end{bmatrix}: (C_{\mK},A_{\mK}) \text{ is observable} \right\}.
\end{equation}
The following lemma is a discrete-time counterpart to~\cite[Lemma 4.5]{zheng2021analysis}. We provide a brief proof in Appendix \ref{sec:proof_positive} for completeness.
\begin{lemma}
\label{lemma.positive_P}
Under Assumption \ref{assumption.control_observe}, if $\mK\in \mathbb{K}\cap\mathbb{K}_o$, the solution $P_{\mK}$ to \eqref{eq.lyapunov_equation} is unique and positive definite.
\end{lemma}

Our next technical result characterizes the structure of the optimal similarity transformation for an observable stabilizing controller. 
\begin{proposition}
\label{proposition.optimal_tansformation}
Suppose $X_{22} \succ 0$ in \eqref{eq:Sigma-X-partition} and $\mK\in \mathbb{K} \cap \mathbb{K}_o$. If an optimal transformation matrix $T^\star \in \mathrm{GL}_n$ exists, then it is unique and in the form of
\begin{equation}
\label{eq.optimal_tranforsmation}
T^\star=- X_{22}X_{12}^{-1}P_{\mK,12}^{-\tr}P_{\mK,22},
\end{equation}
where $P_\mK$, partitioned as~\eqref{eq:Pk-partition}, is the unique positive definite solution to \eqref{eq.lyapunov_equation}.
\end{proposition}

\begin{proof}
From \eqref{eq.cost_transformation}, $J(\mathscr{T}_T(\mK))$ can be written as
\begin{equation}
\label{eq.cost_of_similarity}
\begin{aligned}
J(\mathscr{T}_T(\mK))&={\rm Tr}\left(P_{11}X_{11}+P_{12}T^{-1}X_{12}^\tr \right.\\
&\qquad \left. +P_{12}^\tr X_{12}T^{-\tr}+P_{22}T^{-1}X_{22}T^{-\tr}\right).
\end{aligned}
\end{equation}
For notational convenience, given $\mK\in \mathbb{K}$, we denote the cost value $J(\mathscr{T}_T(\mK))$ w.r.t. the similarity transformation $T$ as 
\begin{equation}
\label{eq.definition_of-H}
g(\mH) := J(\mathscr{T}_T(\mK)), \quad \text{with} \; \mH := T^{-1}\in \mathrm{GL}_n.
\end{equation}
It is clear that $g(\mH)$ is twice differentiable w.r.t. $\mH$. 
The gradient of $g(\mH)$ w.r.t. $\mH$ can be derived as
\begin{equation}
\label{eq.gradient_of_gH}
\begin{aligned}
\nabla_{\mH}g(\mH)=2(P_{12}^\tr X_{12}+P_{22}\mH X_{22}).
\end{aligned}
\end{equation}
By Lemma \ref{lemma.positive_P}, the solution $P_{\mK}$ to \eqref{eq.lyapunov_equation} is positive definite, and hence $P_{22}$ is also invertible. By the assumption of $X_{22} \succ 0$, there is a unique solution to $\nabla_{\mH}g(\mH)=0$, given by
\begin{equation}
\label{eq:Hstar}
    {\mH}^\star = - P_{22}^{-1}P_{12}^\tr X_{12}X_{22}^{-1}.
\end{equation}
If $\mH^\star \in \mathrm{GL}_n$, by $T^\star=({\mH}^\star)^{-1}$, we now identify $T^\star$ is in the form of~\eqref{eq.optimal_tranforsmation}. From \eqref{eq:Hstar}, $\mH^\star \in \mathrm{GL}_n$ requires the invertibility of $X_{12}$ and $P_{12}$. Thus, if an optimal transformation matrix $T^\star \in \mathrm{GL}_n$ exists, both $X_{12}$ and $P_{12}$ must be invertible.

Next, we show that $T^\star$ in~\eqref{eq.optimal_tranforsmation} is the unique globally optimal similarity transformation matrix such that \eqref{eq.optimal_T} holds. We analyze the  Hessian of $g(\mH)$ applied to a nonzero direction $Z \in \mathbb{R}^{n\times n}$, which is  
\begin{equation} 
\nonumber
\nabla^2g(\mH)[Z,Z]:=\frac{d^2}{d\eta^2}\Big|_{\eta=0} g(\mH + \eta Z).
\end{equation}
Using \eqref{eq.cost_of_similarity}, we further have
\begin{equation} 
\nonumber
\begin{aligned}
\,&\nabla^2g(\mH)[Z,Z] \\ 
=\,&\frac{d^2}{d\eta^2}\Big|_{\eta=0}{\rm Tr}(P_{12}(\mH+\eta Z)X_{12}^\tr+P_{12}^\tr X_{12}(\mH+\eta Z)^\tr\\ 
&\qquad\qquad\qquad\quad+P_{22}(\mH+\eta Z)X_{22}(\mH+\eta Z)^\tr)\\
=\,&2{\rm Tr}(P_{22}ZX_{22}Z^\tr)\\
\geq\,& 2\lambda_{\rm min}(P_{22})\lambda_{\rm min}(X_{22})\|Z\|_F^2\\
>\,& 0.
\end{aligned}
\end{equation}
We extend the function $g(\mH)$ to be defined on the convex superset $\mathbb{R}^{n\times n}$ of $\mathrm{GL}_n$. 
It is immediate that $g(\mH)$ is strongly convex over $\mathbb{R}^{n\times n}$. Hence,  the global minimum of $g(\mH)$ over $\mathbb{R}^{n\times n}$ is unique if \cref{eq:Hstar} exists. It follows that ${\mH}^\star$ is also the unique global minimum over $\mathrm{GL}_n$ if it is invertible. Hence, we have proved the optimality and uniqueness of $T^\star$.
\end{proof}

\begin{figure}[t]
\centering
\captionsetup{singlelinecheck = false,labelsep=period, font=small}
\captionsetup[subfigure]{justification=centering}
\subfloat[]{\label{fig:example_1}\includegraphics[width=0.7\linewidth]{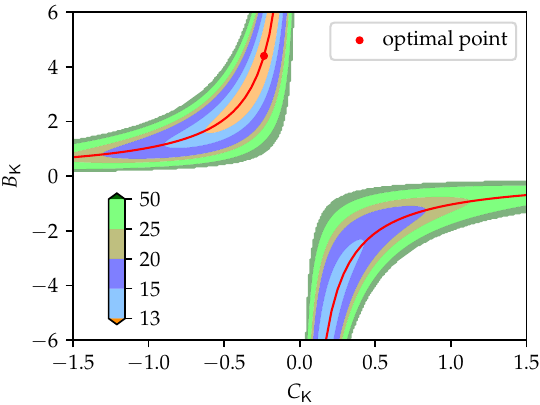}}\\
\subfloat[]{\label{fig:example_2}\includegraphics[width=0.7\linewidth]{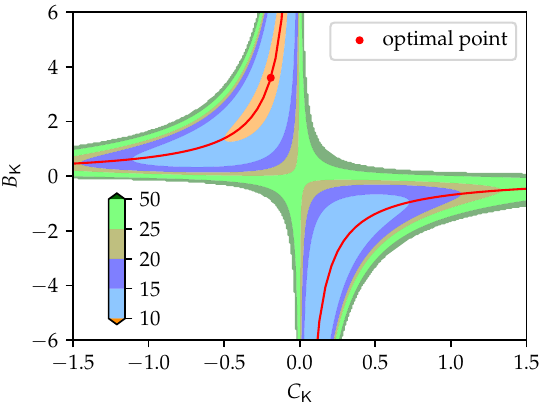}}\\
\caption{The \texttt{dLQR} cost of Examples \ref{example:1} and  \ref{example:2}. (a) \texttt{dLQR} cost in Example \ref{example:1} when fixing $A_{\mK}=-0.944$. The red curve represents all points in the set  $\{(B_{\mK},C_{\mK})|B_{\mK}=1.1T, C_{\mK}=-0.944/T, T\neq 0\}$. (b) \texttt{dLQR} cost in Example \ref{example:2} when fixing $A_{\mK}=-0.765$. The red curve represents all points in the set  $\{(B_{\mK},C_{\mK})|B_{\mK}=0.9T, C_{\mK}=-0.765/T, T\neq 0\}$. Based on an exhaustive numerical grid search, the globally optimal points in Examples \ref{example:1} and \ref{example:2} are located on the red curves in subfigures (a) and (b), respectively. Consequently, the red curves also represent the set of similarity transformations of the globally optimal controller.}
\label{f:problem_1_cost}
\end{figure} 

\Cref{proposition.optimal_tansformation} identifies the form of the optimal similarity transformation, which is unique if it exists. However, it might not exist since $X_{12}$ can be singular.  We give an analytical example in \Cref{appendix.non-existence}.  
%
%
%
%
We conclude this section by providing two examples to demonstrate the impact of similarity transformation on the \texttt{dLQR} cost. 
\begin{example}
\label{example:1}
Consider an open-loop unstable dynamic system \eqref{eq.statefunction} with 
$$A=1.1, \; B=1, \; C=1, \; Q=5, \; R=1.$$ 
The set of stabilizing controllers $\mathbb{K}$ for this system has two disconnected components (see \cite[Theorem D.4, Example 11]{zheng2021analysis}).  For Problem 1, we assume  
\begin{equation} \label{eq:exampleX}
X=\mathbb{E}_{\bar{x}_0\sim \bar{\mathcal{D}}} \; [\bar{x}_0\bar{x}_0^\tr] = \begin{bmatrix}
1&0.25\\
0.25&1
\end{bmatrix}.
\end{equation}
The red curve in Fig. \ref{fig:example_1} shows the orbit of the similarity transformation of the controller
$$\mK = \begin{bmatrix}
0&-0.944\\
1.1&-0.944
\end{bmatrix}.$$
We can see that the \texttt{dLQR} cost changes with different similarity transformations. The optimal similarity transformation corresponds to the red dot in the figure, which demonstrates the result of \Cref{proposition.optimal_tansformation}. \hfill $\square$
\end{example}

\begin{example}
\label{example:2}
Consider an open-loop stable dynamic system \eqref{eq.statefunction} with 
$$A=0.9, \; B=1, \; C=1, \; Q=5, \; R=1.$$ 
The set of stabilizing controllers $\mathbb{K}$ for this system is nonconvex but connected (see \cite[Theorem D.4]{zheng2021analysis}).  To define \texttt{dLQR}~\eqref{eq.dLQR}, we choose the same $X$ as in~\eqref{eq:exampleX}. The red curve in Fig. \ref{fig:example_2} shows the orbit of the similarity transformation of the controller
$$\mK = \begin{bmatrix}
0&-0.765\\
0.9&-0.765
\end{bmatrix}$$
and the red point represents the optimal similarity transformation. As expected, this is consistent with the result in \Cref{proposition.optimal_tansformation}. \hfill $\square$
\end{example}


\section{Gradients and Stationary Points}
\label{sec.mainresults}

In this section, we derive the analytical expression for the gradient of the \texttt{dLQR} cost, 
and characterize the stationary points of Problem \ref{pro.dLQR}. 

\subsection{The Gradient of the \texttt{dLQR} Cost}

The following lemma derives the closed-form formulae of the gradient of the \texttt{dLQR} cost w.r.t. the controller parameters.
\begin{lemma}[Policy Gradient Expression]
\label{lemma:gradient}
 For $\forall  \mK\in \mathbb{K}$, the policy gradient of $J(\mK)$ in Problem \ref{pro.dLQR} is 
\begin{subequations}
\label{eq.gradient}
\begin{align}
&\begin{aligned}
\label{eq.gradient_k12}
\nabla_{C_{\mK}} &J(\mK)=2B^\tr (P_{11}A+P_{12}B_{\mK}C)\Sigma_{12}\\
&\qquad +2((R + B^\tr P_{11} B) C_{\mK}+B^\tr P_{12}A_{\mK})\Sigma_{22},
\end{aligned}\\
&\begin{aligned}
\label{eq.gradient_k21}
\nabla_{B_{\mK}} J(\mK)
&= 2(P_{12}^\tr A +P_{22}B_{\mK}C)\Sigma_{11} C^\tr \\
&\quad + 2(P_{12}^\tr B C_{\mK}+P_{22}A_{\mK})\Sigma_{12}^\tr C^\tr,
\end{aligned}\\
&\begin{aligned}
\label{eq.gradient_k22}
\nabla_{A_{\mK}} J(\mK)&= 2(P_{12}^\tr B C_{\mK}+P_{22}A_{\mK})\Sigma_{22}\\
&\quad + 2 (P_{12}^\tr A+P_{22}B_{\mK}C)\Sigma_{12},
\end{aligned}
\end{align}
\end{subequations}
where $P$ (partitioned as~\eqref{eq:Pk-partition}) and $\Sigma$ (partitioned as~\eqref{eq:Sigma-X-partition}) are the unique positive semidefinite solutions to  \eqref{eq.lyapunov_equation} and \eqref{eq.lyapunov_equation_sigma}, respectively. 
\end{lemma}
\begin{proof}
The proof follows similar derivations as the state-feedback LQR case~\cite[Lemma 1]{fazel2018global}. Using \eqref{eq.lyapunov_equation}, the value function of $\bar{x}_0$ reads as 
\begin{equation}
\nonumber
\begin{aligned}
V_{\mK}(\bar{x}_0) 
= \;&\bar{x}_0^\tr P_{\mK}\bar{x}_0\\
= \;&\bar{x}_0^\tr (\bar{Q} + \bar{C}^\tr \mK^\tr \bar{R}\mK\bar{C})\bar{x}_0\\
&\qquad \quad +\bar{x}_0^\tr(\bar{A}+\bar{B}\mK\bar{C})^\tr P_{\mK}(\bar{A}+\bar{B}\mK\bar{C})\bar{x}_0\\
=\;&\bar{x}_0^\tr (\bar{Q} + \bar{C}^\tr \mK^\tr \bar{R}\mK\bar{C})\bar{x}_0+V_{\mK}((\bar{A}+\bar{B}\mK\bar{C})\bar{x}_0).
\end{aligned}
\end{equation}
Before preceding, we first define a projection operator $\mathcal{T}$ of a $(m+n)\times (d+n)$ matrix $Y$ as
\begin{equation*}
\label{eq.operator}
\mathcal{T}(Y)=Y-\begin{bmatrix}I_m&0_{m\times n}\\0_{n\times m}&0_{n\times n}\end{bmatrix}Y\begin{bmatrix}I_d&0_{d\times n}\\0_{n\times d}&0_{n\times n}\end{bmatrix}.
\end{equation*}

Taking the gradient of $V_{\mK}(\bar{x}_0)$ w.r.t. $\mK$ (note that both $V_{\mK}$ and its argument are functions of $\mK$), we have
\begin{equation}
\nonumber
\begin{aligned}
\nabla_{\mK} V_{\mK}(\bar{x}_0) 
= \; &\mathcal{T}(2E_{\mK}\bar{x}_0\bar{x}_0^\tr\bar{C}^\tr +\bar{x}_1^\tr \nabla_{\mK} P_{\mK} \bar{x}_1\big|_{\bar{x}_1=(\bar{A}+\bar{B}\mK\bar{C})\bar{x}_0})\\
= \; &\mathcal{T}\left(2E_{\mK}\sum_{t=0}^{\infty}\bar{x}_t \bar{x}_t^\tr\bar{C}^\tr\right),
\end{aligned}
\end{equation}
where 
$E_{\mK}:= \bar{R}\mK\bar{C} + \bar{B}^\tr P_{\mK}(\bar{A}+\bar{B}\mK\bar{C})$ 
and the last step follows by recursion and the fact that $\bar{x}_{t+1} = (\bar{A}+\bar{B}\mK\bar{C})\bar{x}_t$.  
Finally, by taking the expectation of the gradients w.r.t. the initial distribution $
\bar{\mathcal{D}}$, we obtain that
\begin{equation*}
\begin{aligned}
\nabla_{\mK} V_{\mK}(\bar{x}_0) 
= 2\mathcal{T}(E_{\mK}\Sigma_{\mK}\bar{C}^\tr),
\end{aligned}
\end{equation*}
which can be partitioned as \eqref{eq.gradient}. This completes the proof. 
\end{proof}

\subsection{Structure of the Observable Stationary Point}
\label{sec.main_result}
We now characterize the stationary points of $J(\mK)$ at which the gradient is zero. Before presenting the main result, we need to state the following proposition on the solution of the Riccati equation, which might be of independent interest.
\begin{proposition}
\label{proposition.solution_of_riccati}
Given an observable pair $(C,A)$, define the set of stabilizing observer gains $\mathbb{L}:=\{L\in \mathbb{R}^{n\times d}:\rho(A-LC)<1\}$. Suppose $C$ has full row rank and $X \succ 0$ is partitioned as in \eqref{eq:Sigma-X-partition}. The following algebraic Riccati equation of $\hat{\Sigma}$ has a unique positive definite solution, 
\begin{equation}
\label{eq.sigma_riccati}
{\hat{\Sigma}}=\Delta_X+A {\hat{\Sigma}} A^\tr - A {\hat{\Sigma}}C^\tr \left(C{\hat{\Sigma}}{C}^\tr\right)^{-1} C {\hat{\Sigma}} A^\tr,
\end{equation}
where 
\begin{equation}
\label{eq.definition_delta_X}
\Delta_X:=X_{11}-X_{12}X_{22}^{-1}X_{12}^\tr. 
\end{equation}
Besides, 
\begin{equation}
\label{eq.optimal_L}
L^\star = A{\hat{\Sigma}}C^\tr (C{\hat{\Sigma}}C^\tr)^{-1} \in \mathbb{L}
\end{equation}
is the unique optimal solution to 
\begin{equation}
\label{eq.problem_of_observer}
\begin{aligned}
 \min_{L\in \mathbb{L}} \quad&{\rm Tr}(\hat{\Sigma}_L)\\
\text{\rm subject to} \quad & \hat{\Sigma}_L=\Delta_X  + (A-LC)\hat{\Sigma}_L(A-LC)^\tr.
\end{aligned}
\end{equation}
\end{proposition}

The proof is given in Appendix \ref{sec.proof_of_riccati}, which is inspired by the convergence analysis of the policy iteration method for LQR \cite[Theorem 1]{hewer1971iterative}. The solution of \eqref{eq.sigma_riccati} is a crucial component in our subsequent analysis on the structure of stationary point. Consider the following canonical discrete-time algebraic Riccati equation, 
\begin{equation}
\label{eq:RiccatiV0}
{\hat{\Sigma}}=\Delta_X+A {\hat{\Sigma}} A^\tr - A {\hat{\Sigma}}C^\tr (C{\hat{\Sigma}}{C}^\tr+V)^{-1} C {\hat{\Sigma}} A^\tr. 
\end{equation}
It is well-known from linear optimal control theory \cite[Proposition 3.1.1]{bertsekas2017dynamic} that the above equation yields a unique positive definite solution when $V\succ 0$ and $(C, A)$ is observable. This is exactly the case considered in \cite[eq. (D.4)]{zheng2021analysis}. However, Proposition \ref{proposition.solution_of_riccati} focuses on the case of $V= 0$, which makes our analysis on stationary point much more complicated than that presented in \cite[Theorem D.4]{zheng2021analysis}. To our knowledge, the characterization of the solution to \eqref{eq:RiccatiV0} with $V = 0$ is not easily accessible in the literature. Proposition \ref{proposition.solution_of_riccati} is thus of independent significance. Besides, it proposes a novel way for finding a stable observer gain by solving \eqref{eq.problem_of_observer}.

Denote the set of stationary points by 
$$
\mathbb{K}_s := \left\{ \begin{bmatrix}
0_{m\times d} & C_{\mK} \\
B_{\mK} & A_{\mK}
\end{bmatrix}:\left\|\begin{bmatrix}
0_{m\times d} & \nabla_{C_{\mK}} J(\mK) \\
\nabla_{B_{\mK}} J(\mK) & \nabla_{A_{\mK}} J(\mK)
\end{bmatrix}\right\|_F=0 \right\}.
$$  
We now investigate the structure of $\mathbb{K}_s$, which is crucial for understanding the performance of policy gradient methods on the \texttt{dLQR} problem. 

\begin{theorem}
\label{theorem.solution_expression}
Suppose $C$ has full row rank, $X \succ 0$, and Assumption \ref{assumption.control_observe} holds. If an observable stationary point $\mK^\star$ (i.e., $\mK^\star \in \mathbb{K}_o \cap \mathbb{K}_s \cap \mathbb{K}$) to Problem \ref{pro.dLQR} exists, it is unique and in the form of 
\begin{equation}
\label{eq.optimal_form_K}
\mK^\star=\mathscr{T}_{T^\star}({\mK}^\ddagger),
\end{equation}
where 
\begin{equation}
\label{eq.riccati_K}
\mK^{\ddagger}:=\begin{bmatrix}
  0   & -K^\star  \\
  L^\star   & A-BK^\star-L^\star C
\end{bmatrix},
\end{equation}
$T^\star$ is the optimal transformation matrix associated with ${\mK}^\ddagger$ computed as 
$ T^\star= X_{22}X_{12}^{-1}$, 
$L^\star$ is defined in \eqref{eq.optimal_L}, and
\begin{equation}
\label{eq.optimal_K}
K^\star= (R+B^\tr {\hat{P}}B)^{-1}B^\tr {\hat{P}}A,
\end{equation}
with ${\hat{P}}$ being the unique positive definite solution to 
\begin{equation}
\label{eq.P_riccati}
{\hat{P}}=Q+ A^\tr {\hat{P}} A -A^\tr {\hat{P}} B (R+B^\tr {\hat{P}} B)^{-1}B^\tr {\hat{P}} A.
\end{equation}
\end{theorem}

\begin{proof}
Let us suppose that there exists an observable stationary point $\mK^\star \in \mathbb{K}_o \cap \mathbb{K}_s \cap \mathbb{K}$. By Lemma \ref{lemma.Lyapunov_stability}(b) and Lemma \ref{lemma.positive_P}, we know $\Sigma_{\mK^\star}, P_{{\mK}^\star}$ are both positive definite. By the Schur complement,  it follows that
\begin{equation} \label{eq:P-tilde}
\begin{aligned}
{\widetilde{P}}&:=P_{11}-P_{12}P_{22}^{-1}P_{12}^\tr \in \mathbb{S}^n_{++},\\  {\widetilde{\Sigma}}&:=\Sigma_{11}  - \Sigma_{12}\Sigma_{22}^{-1}\Sigma_{12}^\tr \in \mathbb{S}^n_{++}.
 \end{aligned}
\end{equation}
For notational convenience, we will omit the subscript of the submatrices of $\Sigma_{\mK^\star}$ and  $P_{{\mK}^\star}$ under the observable stationary point $\mK^\star$ throughout this proof. 

Let \eqref{eq.gradient} equal to 0, and we can solve the linear equations for $A_{\mK}$, $B_{\mK}$, and $C_{\mK}$ to obtain (see Appendix \ref{sec.detail_derivation_of_K_original} for details):\begin{subequations}
\label{eq.K_original}
\begin{align}
&\begin{aligned}
\label{eq.K_12_original}
C_{\mK^\star}& =-K^\star\Sigma_{12}\Sigma_{22}^{-1},
\end{aligned}\\
&\begin{aligned}
\label{eq.K_21_original}
B_{\mK^\star} = -P_{22}^{-1} P_{12}^\tr L^\star,
\end{aligned}\\
&\begin{aligned}
\label{eq.K_22_original}
A_{\mK^\star}=-P_{22}^{-1}P_{12}^\tr (A-BK^\star-L^\star C)\Sigma_{12}\Sigma_{22}^{-1},
\end{aligned}
\end{align}
\end{subequations}
where $K^\star$ and $L^\star$ are
$$
\begin{aligned}
K^\star &= (R + B^\tr{\widetilde{P}} B)^{-1}B^\tr {\widetilde{P}}A, \\
L^\star &= A{\widetilde{\Sigma}} C^\tr (C{\widetilde{\Sigma}} C^\tr)^{-1},
\end{aligned}
$$
where $\widetilde{P}$ and $\widetilde{\Sigma}$ are defined in \eqref{eq:P-tilde}.

Combining  \eqref{eq.block_lyapunov_P12}, \eqref{eq.block_lyapunov_P22}, and \eqref{eq.K_original}, it can be further shown that (detailed calculations are provided in Appendix \ref{sec.detail_derivation_of_inverse})
\begin{equation}
\label{eq.inverse}
 P_{12}^\tr \Sigma_{12} + P_{22}\Sigma_{22}=0,  
\end{equation}
and hence,
\begin{equation}
\label{eq.inverse_relation}
(-P_{22}^{-1}P_{12}^\tr)^{-1}=\Sigma_{12}\Sigma_{22}^{-1}.
\end{equation}
We then define $T^\ddagger:=-P_{22}^{-1}P_{12}^\tr$, and thus $(T^\ddagger)^{-1}=\Sigma_{12}\Sigma_{22}^{-1}$. 

The transformation matrix $T^\ddagger$, however, still depends on $\mK^\star$. Unlike \cite[Theorem D.4]{zheng2021analysis}, the cost of Problem \ref{pro.dLQR} varies with different similarity transformations. Therefore, it is necessary to decouple $T^\ddagger$ from $\mK^\star$ to make expression \eqref{eq.K_original} explicit. From \eqref{eq.block_lyapunov_sigma12}, \eqref{eq.block_lyapunov_sigma22}, and  \eqref{eq.K_original}, equation \eqref{eq.inverse} can be rewritten as 
\begin{equation}
\label{eq.relation_P_and_x}
P_{12}^\tr X_{12} + P_{22}X_{22}=0.
\end{equation}
See Appendix \ref{sec.detail_derivation_of_T} for details on deriving \eqref{eq.relation_P_and_x}. Combining \eqref{eq.inverse} with \eqref{eq.relation_P_and_x} leads to
\begin{equation}
\label{eq.T_and_X}
{T^\ddagger}= -P_{22}^{-1}P_{12}^\tr=X_{22}X_{12}^{-1},
\end{equation}
which depends solely on the initial distribution $\bar{\mathcal{D}}$.

Based on \eqref{eq.K_original}, \eqref{eq.inverse_relation}, and \eqref{eq.T_and_X}, we can see that $\mK^\star$ is in the form shown in \eqref{eq.optimal_form_K}. It remains to show that 
\begin{itemize}
    \item ${T^\ddagger}$ is the optimal transformation matrix of ${\mK}^\ddagger$ (i.e., $T^\ddagger=- X_{22}X_{12}^{-1}P_{{\mK}^\ddagger,12}^{-\tr}P_{{\mK}^\ddagger,22}=T^\star$, see \eqref{eq.optimal_tranforsmation});
    \item $\widetilde{P}={\hat{P}}$ and $\widetilde{\Sigma}={\hat{\Sigma}}$, that is, they are the unique positive definite solutions to the Riccati equations \eqref{eq.P_riccati} and \eqref{eq.sigma_riccati}, respectively. 
\end{itemize}
First,  by \eqref{eq.optimal_form_K} and \eqref{eq.similarity_P} in Proposition \ref{prop.similarity_variance}, we have
\begin{equation}
\label{eq.partition_of_PK}
P_{{\mK}^\star}=\begin{bmatrix}
I_n & 0 \\
0 & (T^\ddagger)^{-\tr}
\end{bmatrix}P_{{\mK}^\ddagger}\begin{bmatrix}
I_n & 0 \\
0 & (T^\ddagger)^{-1}
\end{bmatrix}.
\end{equation}
Plugging the expression of $P_{{\mK}^\star}$ (see \eqref{eq.partition_of_PK}) in \eqref{eq.relation_P_and_x}, it is not hard to show that
\begin{equation}
\label{eq.transformation_equality}
(T^\ddagger)^{-\tr}P_{{\mK}^\ddagger,12}^\tr X_{12} + (T^\ddagger)^{-\tr}P_{{\mK}^\ddagger,22}(T^\ddagger)^{-1}X_{22}=0.
\end{equation}
Using \eqref{eq.T_and_X} in \eqref{eq.transformation_equality}, it directly leads to $$
-P_{{\mK}^\ddagger,22}^{-1}P_{{\mK}^\ddagger,12}^\tr=I_n.
$$
Therefore, by \eqref{eq.optimal_tranforsmation} of  \Cref{proposition.optimal_tansformation}, one has
\begin{equation}
\nonumber
{T^\ddagger}=X_{22}X_{12}^{-1}=- X_{22}X_{12}^{-1}P_{{\mK}^\ddagger,12}^{-\tr}P_{{\mK}^\ddagger,22}=T^\star,
\end{equation}
which is exactly the optimal transformation matrix of ${\mK}^\ddagger$.

Then, we will derive $\widetilde{P}=\hat{P}$.  Multiplying \eqref{eq.block_lyapunov_P22} by ${T^\star}^\tr$ on the left and by ${T^\star}$ on the right (or multiplying \eqref{eq.block_lyapunov_P12} by ${T^\star}$ on the right), we have
\begin{equation}
\label{eq.P_22_transform}
\begin{aligned}
&P_{12}P_{22}^{-1}P_{12}^\tr= A^\tr {\widetilde{P}} B (R+B^\tr {\widetilde{P}} B)^{-1}B^\tr {\widetilde{P}} A \\
&\quad +A^\tr P_{12}P_{22}^{-1}P_{12}^\tr A+C^\tr {L^\star}^\tr P_{12}P_{22}^{-1}P_{12}^\tr L^\star C\\
&\quad -A^\tr P_{12}P_{22}^{-1}P_{12}^\tr L^\star C-C^\tr {L^\star}^\tr P_{12}P_{22}^{-1}P_{12}^\tr A.
\end{aligned}
\end{equation}
Then,  plugging \eqref{eq.K_21_original} in \eqref{eq.block_lyapunov_P11} leads to
\begin{equation}
\label{eq.P_11_expand}
\begin{aligned}
P_{11}&=Q+ A^\tr P_{11} A +C^\tr {L^\star}^\tr P_{12}  P_{22}^{-1}P_{12}^\tr L^\star C\\
&- A^\tr P_{12} P_{22}^{-1}P_{12}^\tr L^\star C- C^\tr {L^\star}^\tr P_{12} P_{22}^{-1} P_{12}^\tr A .
\end{aligned}
\end{equation}
Subtracting \eqref{eq.P_22_transform} from \eqref{eq.P_11_expand}, we can finally see that $\widetilde{P}=\hat{P}$ satisfying the Riccati equation \eqref{eq.P_riccati}.

Through
similar steps, we can derive from \eqref{eq.block_lyapunov_sigma} that $\widetilde{\Sigma}=\hat{\Sigma}$ satisfying the Riccati equation \eqref{eq.sigma_riccati}. By Proposition \ref{proposition.solution_of_riccati}, we further know that \eqref{eq.sigma_riccati} only admits a unique positive definite solution. This completes the proof.
\end{proof}

Note that the positive definiteness of $P_{\mK}$ and $\Sigma_{\mK}$ were utilized in the proof Theorem \ref{theorem.solution_expression}. By Lemma \ref{lemma.positive_P}, we observe that $P_{\mK}\in \mathbb{S}_{++}^{2n}$ if $\mK\in \mathbb{K}_o$; by Lemma \ref{lemma.Lyapunov_stability}(b) and Lemma \ref{lemma.positive_P}, $\Sigma_{\mK}\in \mathbb{S}_{++}^{2n}$ if $X\succ 0$ or $(\bar{A}+\bar{B}\mK\bar{C},W)$ is reachable, where $X=WW^\tr$. Compared to the reachable condition that relies on the system dynamics, $\mK$, and $X$, the positive definiteness of $X$ can be achieved more easily by carefully designing a proper initial controller state distribution. Therefore, we assume $X \succ 0$ in Proposition \ref{proposition.solution_of_riccati} and Theorem \ref{theorem.solution_expression}. Similar to the assumption of $\mathbb{E}_{x_0\sim \mathcal{D}}[x_0x_0^\tr] \succ 0$ for optimizing the full state-feedback gain \cite{fazel2018global,jansch2020Mjump}, the condition $X \succ 0$ can be informally thought as the persistent excitation condition for the augmented system \eqref{eq.closed-loop-system_short}.

Theorem~\ref{theorem.solution_expression} reveals that the observable stationary point~$\mK^\star$ has an elegant closed-form: it is the optimal similarity transformation of a special observer-based controller $\mK^\ddagger$. In particular, $K^\star$ of \eqref{eq.riccati_K} is exactly the optimal control gain of the state-feedback LQR and $L^\star$ is a stable observer gain. In linear optimal control theory \cite{lewis2012optimal}, the observer-based controller consists of a stable observer and  a state-feedback LQR, which are designed separately; however, the transient behavior induced by the initial system and controller states is not considered. In the learning context of the \texttt{dLQR} formulation, the dynamic controller is learned as a whole, with initial states sampled from $\bar{\mathcal{D}}$ in practice. In the analysis, \texttt{dLQR} cost depends on $\bar{\mathcal{D}}$, and thus both the observer gain $L^\star$ and the optimal~transformation matrix ${T^\star}$ in \eqref{eq.optimal_form_K} are uniquely determined by $\bar{\mathcal{D}}$. 

In practical applications, if the optimal controller of a given system is known to exist and be observable, then $\mK^\star$ in~\eqref{eq.optimal_form_K} must be the globally optimal controller due to its uniqueness. For instance, the observable stationary points of Examples \ref{example:1} and \ref{example:2} are
$$\mK_1^\star = \begin{bmatrix}
0&-0.236\\
4.4& -0.944
\end{bmatrix} \quad \text{and} \quad \mK_2^\star = \begin{bmatrix}
0&-0.191\\
3.6& -0.765
\end{bmatrix},$$ 
and they agree with the exhaustive numerical grid search for the globally optimal points (marked as red points in Fig. \ref{f:problem_1_cost}) in Examples \ref{example:1} and \ref{example:2}, respectively. As described in Fig. \ref{f:problem_1_cost}, the red curves also represent the set of similarity transformations of the globally optimal controller.

Next, we discuss a special case of \Cref{theorem.solution_expression} when we have perfect knowledge of the initial system state, i.e., $\xi_0 = x_0$. We can show that the observable stationary point $\mK^\star$ in \eqref{eq.optimal_form_K} is globally optimal for \texttt{dLQR}, yielding control performance equal to the optimal full state-feedback LQR. Note that the positive definiteness of $\Sigma_{\mK^\star}$ was utilized in the establishment of \Cref{theorem.solution_expression}. Since $\xi_0=x_0$, $X =  \mathbb{E}_{\bar{x}_0\sim \bar{\mathcal{D}}} \; [\bar{x}_0\bar{x}_0^\tr]$ is not positive definite. In this case, by \Cref{lemma.Lyapunov_stability}(c), the reachability of $(\bar{A}+\bar{B}\mK^\star\bar{C},X^{\frac{1}{2}})$ is required to guarantee the positive definiteness of $\Sigma_{\mK^\star}$.

\begin{proposition}[Equivalence between \texttt{dLQR} and state-feedback LQR] 
\label{proposition.equivalent_to_LQR}
Suppose $C$ has full row rank, $X_{11}=\mathbb{E}_{x_0\sim \mathcal{D}} [x_0x_0^\tr] \succ 0$, and Assumption \ref{assumption.control_observe} holds. The observable stationary point (i.e., $\mK^\star \in \mathbb{K}_o \cap \mathbb{K}_s \cap \mathbb{K}$) to Problem \ref{pro.dLQR} with $\xi_0=x_0$ is in the form of \eqref{eq.optimal_form_K} with $T^\star=I_n$, which is globally optimal for \texttt{dLQR}, yielding the same control performance as the optimal state-feedback LQR. 
\end{proposition}

\begin{proof}
Given $\xi_0 = x_0$, we have $X_{12} = X_{22}= X_{11}$. Because ${X_{11}}^{\frac{1}{2}} \in \mathbb{S}_{++}^n$, $(A, {X_{11}}^{\frac{1}{2}})$ and $(A_{\mK^\star}, {X_{11}}^{\frac{1}{2}})$ are both reachable. Therefore, similar to Lemma \ref{lemma.positive_P}, it is easy to show that $\Sigma_{\mK^\star} \in \mathbb{S}_{++}^{2n}$. Then following the similar proof steps for  \Cref{theorem.solution_expression}, we can easily show that the observable stationary point to Problem \ref{pro.dLQR} given $\xi_0=x_0$ is in the form of \eqref{eq.optimal_form_K}. 

By inserting \eqref{eq.optimal_form_K} with $T^\star= X_{22}X_{12}^{-1}=I_n$ into \eqref{eq.dynamic_controller}, the dynamic controller reads \eqref{eq.optimal_form_K} as
$$
\begin{aligned}
u_0 &= -K^\star\xi_0, \\
u_1 &= -K^\star\left( (A-BK^\star)\xi_0 + L^\star C(x_0-\xi_0) \right),\\
&\vdots\\
u_t &= -K^\star\Big( (A-BK^\star)^t\xi_0 \\
&\quad+\sum_{k=0}^{t-1} (A-BK^\star)^kL^\star C(A-L^\star C)^{t-1-k}(x_0-\xi_0) \Big).
\end{aligned}
$$
Since $\xi_0 = x_0$, we now get 
$$\begin{bmatrix}
 u_0\\
 u_1\\
 \vdots\\
 u_t
\end{bmatrix}= \begin{bmatrix}
-K^\star x_0\\
-K^\star(A-BK^\star)x_0\\
 \vdots\\
-K^\star (A-BK^\star)^t x_0
\end{bmatrix},$$
which is equivalent to the globally optimal control sequence of the state-feedback LQR. It is clear that this result holds for any observer gain $L$. This completes the proof.
\end{proof}

The result of Theorem \ref{theorem.solution_expression} is important since it provides a certificate of optimality for policy gradient methods. In particular, this allows us to check whether the converged point of policy gradient methods is globally optimal to Problem~\ref{pro.dLQR} under moderate assumptions.

\begin{corollary}
\label{corollary.globle_optimal}
Given an LTI system \eqref{eq.statefunction}, suppose that the globally optimal controller of Problem \ref{pro.dLQR} is observable (i.e., $(C_{\mK},A_{\mK}) \text{ is observable}$). Consider a policy gradient algorithm $\mK_{i+1}=\mK_{i}-\alpha_i \nabla_{\mK_i} J(\mK_i)$, where $\alpha_i>0$ is an appropriate learning rate such that $\mathsf{K}_i \in \mathbb{K}$, $\forall i\geq 0$ and $\inf_i \alpha_i >0$. If the algorithm converges to an observable stationary point $\mK^\star \in\mathbb{K}_s \cap \mathbb{K}_o$, then $\mK^\star$ is globally optimal.
\end{corollary}

Since the full-order dynamic controller \eqref{eq.dynamic_controller} does not depend on
system parameters, our results
can also be generalized to the model-free case. In the model-free setting, existing policy-based learning techniques, such as the zeroth-order optimization approach, provide an effective way to obtain an unbiased estimate of the policy gradient from sample trajectories \cite{conn2009introduction,nesterov2017random,fazel2018global}. Note that Corollary \ref{corollary.globle_optimal} does not discuss under what conditions will the gradient descent iterates converge. The convergence of model-based or model-free policy gradient methods is left for future work.

Finally, we highlight that the observable stationary point of Problem \ref{pro.dLQR} does not exist when $X_{12}$ is singular. 

\begin{corollary}
\label{corollary.exist_of_solution}
The observable stationary point of Problem \ref{pro.dLQR} exists only if $X_{12}$ is invertible.
\end{corollary}
This can be seen from the proof of  \Cref{theorem.solution_expression}. More specifically, it follows from \eqref{eq.T_and_X}. The result indicates that designing an initial controller state $\xi_0$ correlated with the initial system state $x_0$ will facilitate learning the dynamic controller.

\section{Equivalence between \texttt{dLQR} and LQG}
\label{sec:lqr_lqg}

In this section, we show the equivalence between the optimal solutions of \texttt{dLQR} and LQG when the initial controller state $\xi_0$ in~\eqref{eq.dynamic_controller} satisfies a certain structural constraint. 

\subsection{Equivalence Analysis} \label{subsetion:problem-2}
Consider the following parameterization of $\xi_0$,
\begin{equation}
\label{eq.structure_constraint}
\xi_0=B_{\mK}s, 
\end{equation}
where $s\in \mathbb{R}^d$ is a random vector. To optimize both the dynamic controller and initial controller state, we provide a variant of the \texttt{dLQR} problem as follows. 
\begin{problem}[Policy optimization for \texttt{dLQR} when $\xi_0$ is a function of $\mK$]
\label{pro.dLQR_initial_estimate}
\begin{equation}  
\nonumber
\begin{aligned}
\min_{\mK} \quad&  J(\mK)\\
\text{\rm subject to} \quad &\mK\in \mathbb{K},
\end{aligned}
\end{equation}
where $J(\mK)$ is defined in~\eqref{eq.cost_in_P} and $\mathbb{K}$ is given in~\eqref{eq:stabilizing-K}. The initial controller state $\xi_0$ here satisfies the structural constraint \eqref{eq.structure_constraint}, where $s\in \mathbb{R}^d$ is randomly sampled from the distribution $\mathcal{D}_s$, which has zero-mean and is independent of the initial system state $x_0$. In this case, $X= \mathbb{E}_{\bar{x}_0\sim \bar{\mathcal{D}}}[\bar{x}_0\bar{x}_0^\tr]= \begin{bmatrix}
   X_{11}  & 0 \\
    0 & B_{\mK} V B_{\mK}^\tr
\end{bmatrix}$, where $V = \mathbb{E}_{s\sim \mathcal{D}_s}[ss^\tr]$ is fixed. 
\end{problem}

Although Problem \ref{pro.dLQR_initial_estimate} is formulated based on deterministic LTI systems, we will show that it is equivalent to the canonical LQG problem. Consider a discrete-time stochastic LTI system,
\begin{equation}   
\label{eq.lqg_system}
\begin{aligned}
x_{t+1} &= Ax_t+Bu_t+w_t,\\
y_t &= Cx_t+v_t,
\end{aligned}
\end{equation}
where $w_t \in \mathbb{R}^n$, $v_t \in \mathbb{R}^d$ represent system process and measurement noises. It is assumed that $w_t$ and $v_t$ are independent white Gaussian noises with intensity matrices $X_{11}$ and $V$. For completeness, we present the classical LQG problem, which is as follows.
\begin{problem}[Policy optimization for LQG]
\label{pro.LQG}
\begin{equation}  
\begin{aligned}
\nonumber
\min_{\mK} \quad & \lim_{N\rightarrow \infty} \frac{1}{N} \mathbb{E}_{x_0,w_t,v_t}\Big[\sum_{t=0}^{N-1}(x_t^\tr Qx_t+u_t^\tr Ru_t) \Big]\\
\text{\rm subject to}\quad  & \eqref{eq.lqg_system}, ~ \eqref{eq.dynamic_controller}, ~\mK\in \mathbb{K}.
\end{aligned}
\end{equation}
\end{problem}

It is clear that the LQG objective in Problem~\ref{pro.LQG} is an average cost in an infinite-time horizon $N \to \infty$, which focuses on the steady-state covariance only, i.e.,
$$
    \mathbb{E}_{x_0,w_t,v_t}\Big[x_\infty^\tr Qx_\infty+u_\infty^\tr Ru_\infty \Big].
$$
The transient behavior is neglected in the classical LQG problem. Instead, the \texttt{dLQR}~\eqref{eq.dLQR}  minimizes an infinite-horizon accumulated cost, in which the system transient behavior induced by initial system and controller states play an important role, as characterized in Lemma \ref{lemma.dLQR_cost}. 

\begin{proposition}
\label{proposition.identity}
If $X_{11} = \mathbb{E}[w_tw_t^\tr]$ and $V = \mathbb{E}[v_tv_t^\tr]$, then Problems \ref{pro.dLQR_initial_estimate} and \ref{pro.LQG} are equivalent in the sense that they have the same optimal solutions.
\end{proposition}
\begin{proof}
From the definition of Problem \ref{pro.dLQR_initial_estimate} and the characterization of the cost function for the LQG problem in \cite[Lemma D.1]{zheng2021analysis}, the cost function in Problem \ref{pro.dLQR_initial_estimate} (see \Cref{lemma.dLQR_cost}) is the same as the LQG cost in Problem~\ref{pro.LQG}. Also, they have the same feasible region, as characterized by the set of stabilizing controllers in \eqref{eq:stabilizing-K}. Therefore, Problems \ref{pro.dLQR_initial_estimate} and \ref{pro.LQG} are equivalent and they have the same optimal solutions.
\end{proof}

We note that the equivalence between Problem \ref{pro.dLQR_initial_estimate} and the corresponding LQG problem reveals the interesting correspondence between the optimal control for deterministic and stochastic LTI systems.

\subsection{Structure of Minimal Stationary Points}
We refer to \eqref{eq.dynamic_controller} as a  minimal  controller if it is a minimal realization of its transfer function. This is equivalent to the case that \eqref{eq.dynamic_controller} is a reachable and observable system. We denote the set of minimal controllers as 
$$
\mathbb{K}_m := \left\{ \begin{bmatrix}
0_{m\times d} & C_{\mK} \\
B_{\mK} & A_{\mK}
\end{bmatrix}: \begin{aligned} &(C_{\mK},A_{\mK}) \text{ is observable}  \\ &(A_{\mK}, B_{\mK}) \text{ is reachable} \end{aligned} \right\}.
$$

Different from Problem \ref{pro.dLQR},  $X$ is only required to be positive semidefinite in Problem \ref{pro.dLQR_initial_estimate} (since both $X_{11}$ and $B_{\mK}VB_{\mK}^\tr$ can be of low rank). Therefore, similar to Lemma \ref{lemma.positive_P}, the reachability of $(A_{\mK}, B_{\mK})$ is required to guarantee the positive definiteness of $\Sigma_{\mK}$. This means that we have $P_{\mK},\; \Sigma_{\mK}\in \mathbb{S}_{++}^{2n}$ if $\mK\in \mathbb{K}_m$, which will be utilized in the following analysis.

\begin{theorem}
\label{theorem.optimal_equivalence}
Suppose $C$ has full row rank,  $(A, {X_{11}}^{\frac{1}{2}})$ is reachable, $V \in \mathbb{S}_{++}^d$, and Assumption \ref{assumption.control_observe} holds. All minimal stationary points $\mK^\star \in  \mathbb{K}\cap\mathbb{K}_m \cap \mathbb{K}_s $ to Problem \ref{pro.dLQR_initial_estimate} are globally optimal, and they are in the form of 
\begin{equation}
\label{eq.optimal_K_LQG}
{\mK}^\star=\mathscr{T}_T({\mK}^\ddagger),
\end{equation}
where 
\begin{equation}
\label{eq.dlqr_2_cost}
\mK^{\ddagger}:=\begin{bmatrix}
  0   & -K^\star  \\
  L^\star   & A-BK^\star-L^\star C
\end{bmatrix},
\end{equation}
$T\in \mathrm{GL}_n$ is an arbitrary invertible matrix, $K^\star$ is defined in \eqref{eq.optimal_K}, and
\begin{equation}
\nonumber
L^\star = A{\hat{\Sigma}}C^\tr (C{\hat{\Sigma}}C^\tr + V)^{-1},
\end{equation}
with ${\hat{\Sigma}}$ being the unique positive definite solution to the following Riccati equation 
\begin{equation}
\label{eq.riccati_KF}
{\hat{\Sigma}}=X_{11}+A {\hat{\Sigma}} A^\tr - A {\hat{\Sigma}}C^\tr (C{\hat{\Sigma}}\bar{C}^\tr+ V)^{-1} C {\hat{\Sigma}} A^\tr.
\end{equation}
\end{theorem}
\begin{proof}
The key point of this proof is that the gradient of the cost of Problem \ref{pro.dLQR_initial_estimate} w.r.t. $B_{\mK}$, i.e.,  $\nabla_{B_{\mK}} J(\mK)$, is different from that of Problem \ref{pro.dLQR}. 

In particular, for Problem \ref{pro.dLQR_initial_estimate}, we get
\begin{equation}
\nonumber
\begin{aligned}
\nabla_{B_{\mK}} J(\mK)
&= 2\left(P_{12}^\tr A \Sigma_{11} C^\tr  +P_{22}B_{\mK}(C\Sigma_{11} C^\tr +V)\right)\\
&\quad + 2(P_{12}^\tr B C_{\mK}+P_{22}A_{\mK})\Sigma_{12}^\tr C^\tr.
\end{aligned}
\end{equation}
Similar to Lemma \ref{lemma.positive_P}, when $(A, {X_{11}}^{\frac{1}{2}})$ and $(A_{\mK^\star}, B_{\mK^\star})$ are both reachable, $\Sigma_{\mK^\star} \in \mathbb{S}_{++}^{2n}$. Following similarly to the proof of Theorem \ref{theorem.solution_expression}, we can show that all minimal stationary points in Problem \ref{pro.dLQR_initial_estimate} are in the form of \eqref{eq.optimal_K_LQG}. Since the controller from the Riccati equations \eqref{eq.P_riccati} and \eqref{eq.riccati_KF} gives the globally optimal LQG controller \cite{zhou1996robust}, we complete the proof by invoking the equivalence in \Cref{proposition.identity}. 
\end{proof}

Note that the observer gain $L^\star$ in this theorem equals the Kalman gain of discrete-time LQG. Theorem \ref{theorem.optimal_equivalence} suggests that Problem \ref{pro.dLQR_initial_estimate} enjoys good properties of symmetry induced by similarity transformations and global optimality of minimal stationary points. Different from the results in Theorem \ref{theorem.solution_expression}, the minimal stationary points of Problem \ref{pro.dLQR_initial_estimate} are not unique, and these points are identical up to a similarity transformation. The main reason for the difference is that the structure of the initial controller state of Problem \ref{pro.dLQR_initial_estimate} induces the invariance of the cost under the similarity transformations.  Thus utilizing the existing LQG results  \cite{zhou1996robust}, it can be further shown that all minimal stationary points of Problem \ref{pro.dLQR_initial_estimate} are globally optimal. 

We provide Examples \ref{example:3} and \ref{example:4} to illustrate the \texttt{dLQR} cost under the setting of Problem \ref{pro.dLQR_initial_estimate}, which show the invariance of the cost under similarity transformations.  

\begin{figure}[t]
\centering
\captionsetup{singlelinecheck = false,labelsep=period, font=small}
\captionsetup[subfigure]{justification=centering}
\subfloat[]{\label{fig:example_3}\includegraphics[width=0.7\linewidth]{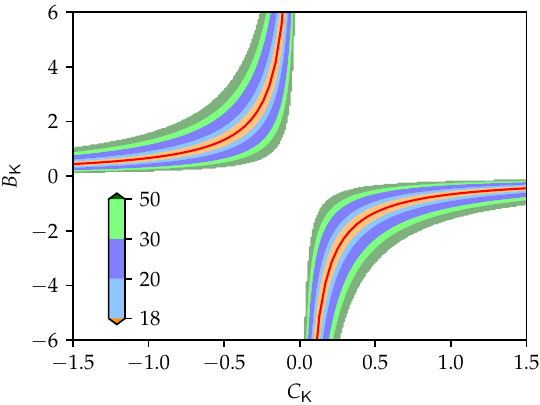}}\\
\subfloat[]{\label{fig:example_4}\includegraphics[width=0.7\linewidth]{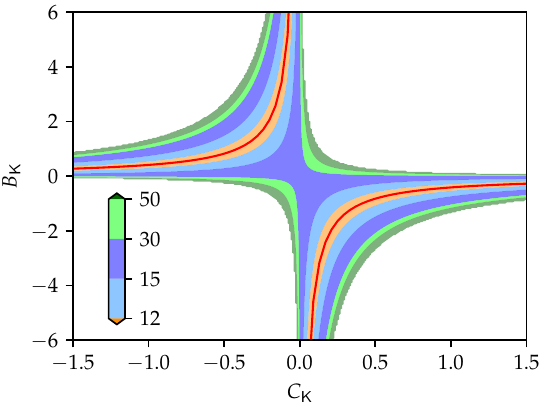}}\\
\caption{\texttt{dLQR} cost of Examples \ref{example:3} and \ref{example:4}. (a) \texttt{dLQR} cost for system in Example \ref{example:3} when fixing $A_{\mK}=-0.547$. The red curve represents the set of globally optimal points  $\{(B_{\mK},C_{\mK})|B_{\mK}=0.703T, C_{\mK}=-0.944/T, T\neq 0\}$. (b) \texttt{dLQR} cost for system in Example \ref{example:4} when fixing $A_{\mK}=-0.403$. The red curve represents the set of globally optimal points  $\{(B_{\mK},C_{\mK})|B_{\mK}=0.538T, C_{\mK}=-0.765/T, T\neq 0\}$. }
\label{f:cost_lqg}
\end{figure}

\begin{example}
\label{example:3}
Consider the system in Example \ref{example:1}.  To define \texttt{dLQR}~\eqref{eq.dLQR}, we choose  
\begin{equation} \label{eq:exampleX2}
X=\mathbb{E}_{\bar{x}_0\sim \bar{\mathcal{D}}} \; [\bar{x}_0\bar{x}_0^\tr] = \begin{bmatrix}
1&0\\
0&B_{\mK} V B_{\mK}^\tr
\end{bmatrix},
\end{equation}
with $V=1$. Theorem~\ref{theorem.optimal_equivalence} implies that all minimal stationary points are globally optimal, which are identical up to a similarity transformation. Fig. \ref{fig:example_3} demonstrates this fact. \hfill $\square$ 
\end{example}

\begin{example}
\label{example:4}
Consider the system in Example \ref{example:2}. To define \texttt{dLQR}~\eqref{eq.dLQR}, we choose $X$ as~\eqref{eq:exampleX2}. Again, all minimal stationary points are globally optimal, which are identical up to a similarity transformation, shown in Fig. \ref{fig:example_4}.    \hfill $\square$
\end{example}

We can also extend Theorem \ref{theorem.optimal_equivalence} to a more general case where $s$ in \eqref{eq.structure_constraint} can be correlated with $x_0$; interested readers can refer to Appendix \ref{appendix.general_equi} for details. We finally provide the following remark highlighting the importance of initial controller states for~\eqref{eq.objective} when using dynamic output-feedback policies.

\begin{remark}[Design of initial controller states]
In practice, if a correlated initial controller state, such that the cross-correlation matrix $X_{12}=\mathbb{E}[x_0\xi_0^\tr]$ is nonsingular, can be obtained based on the prior information and output observation, Problem \ref{pro.dLQR} usually yields a better dynamic controller than Problem \ref{pro.dLQR_initial_estimate}. For example, Examples \ref{example:1} and \ref{example:3} share the same system parameters and initial system state correlation ($X_{11}=1$) for~\eqref{eq.objective}, while the minimum cost of Example \ref{example:1} (which is 11.914) is smaller than Example \ref{example:3} (17.156). Similarly, Example \ref{example:2} (9.363) has a smaller minimum cost than Example \ref{example:4} (11.504). On the other hand, if the prior information for the design of the initial controller state is limited, Problem \ref{pro.dLQR_initial_estimate} might be more suitable considering the global optimality of minimal stationary points and the invariance property of similarity transformations.
\hfill $\square$
\end{remark}

\section{Numerical Experiments}
\label{sec:numerical_experi}
We have illustrated our main results on the structure of stationary controllers in previous sections, which are crucial for establishing a certificate of optimality for policy gradient methods. Here, we present some numerical experiments to demonstrate the empirical performance of policy gradient methods for solving the \texttt{dLQR} problem under the setting of Problems \ref{pro.dLQR} and \ref{pro.dLQR_initial_estimate}.

We consider the vanilla policy gradient method (known as the gradient descent method). 
As described in Corollary \ref{corollary.globle_optimal}, upon giving an initial stabilizing controller
$\mK_0\in\mathbb{K}$, we update the controller using
\begin{equation}
\label{eq.update_rule}
\mK_{i+1}=\mK_{i}-\alpha_i \nabla_{\mK_i} J(\mK_i),
\end{equation}
until the gradient satisfies $\|\nabla_{\mK_i} J(\mK_i)\|_F\le \epsilon$ or the algorithm reaches $i_{\rm max}$ iterations \footnote{The code for numerical experiments is available at \url{https://github.com/soc-ucsd/LQG_gradient/tree/master/dLQR}}. 
Similar to~\cite{zheng2021analysis}, the learning rate $\alpha_i$ is determined by the Armijo rule \cite[Chapter 1.3]{bertsekas1997nonlinear}: Set $\alpha_i=1$, repeat $\alpha_i=\beta\alpha_i$ until 
$$J(\mK_i)-J(\mK_{i+1})\ge \theta\alpha_i\|\nabla_{\mK_i} J(\mK_i)\|_F^2,$$
where $\beta \in (0,1), \; \theta\in (0,1)$.  In this paper, we set $\beta=0.5$ and $\theta=0.01$.

\begin{figure}[t]
\centering
\captionsetup{singlelinecheck = false,labelsep=period, font=small}
\captionsetup[subfigure]{justification=centering}
\subfloat[]{\label{fig:curve_p1}\includegraphics[width=0.5\linewidth]{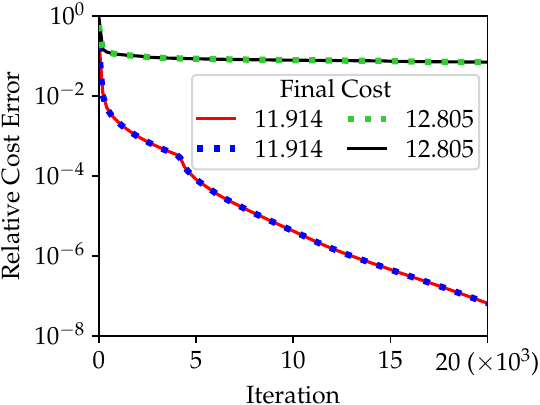}}
\subfloat[]{\label{fig:curve_p2}\includegraphics[width=0.5\linewidth]{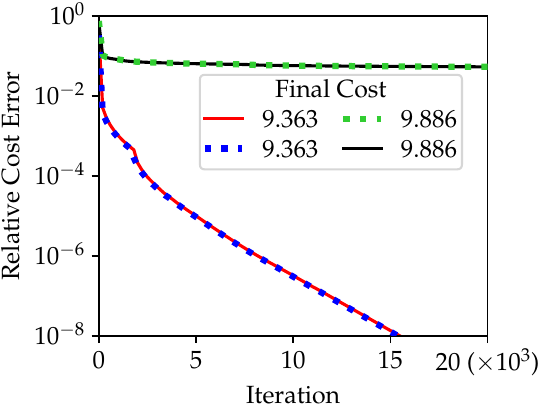}}
\\
\subfloat[]{\label{fig:curve_p3}\includegraphics[width=0.5\linewidth]{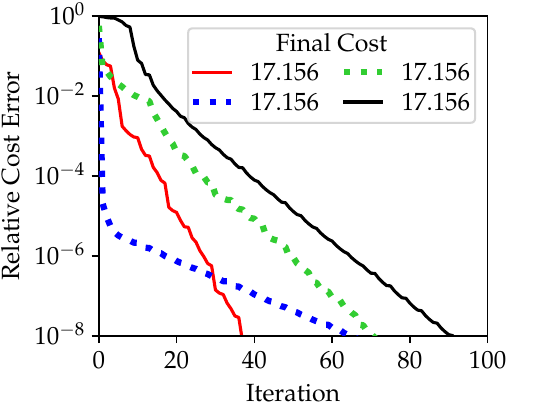}}
\subfloat[ ]{\label{fig:curve_p4}\includegraphics[width=0.5\linewidth]{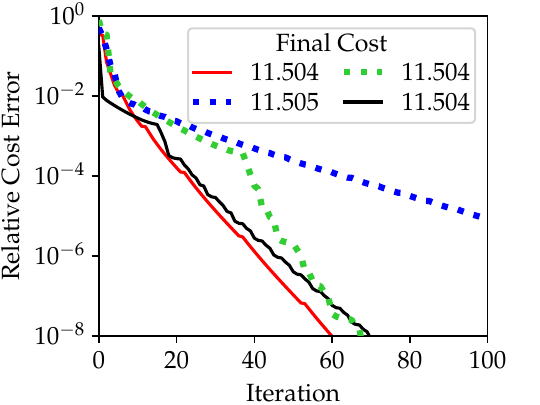}}
\caption{Learning curves of Examples \ref{example:1}-\ref{example:4} with four different random initialization (corresponding to curves with different colors). (a) Learning curves of Example \ref{example:1} with $J({\mK}^\star)=11.914$ and $i_{\max}=2\times 10^4$. (b) Learning curves of Example \ref{example:2} with $J({\mK}^\star)=9.363$  and $i_{\max}=2\times 10^4$. (c) Learning curves of Example \ref{example:3} with $J({\mK}^\star)=17.156$ and $i_{\max}=100$.  (d) Learning curves of Example \ref{example:4} with $J({\mK}^\star)=11.504$ and $i_{\max}=100$.}
\label{f:learning_curve}
\end{figure}

\subsection{Performance on Examples \ref{example:1}-\ref{example:4}}
\label{sec:result_e1-4}
Fig. \ref{f:learning_curve} (a)-(d) show the normalized cost error during the learning process of Examples \ref{example:1}-\ref{example:4}, which is computed as $|J(\mK)-J(\mK^{\star})|/J(\mK)$. Recall that Examples \ref{example:1} and \ref{example:3} are two different characterizations of the original \texttt{dLQR} problem \eqref{eq.dLQR} with the same dynamics, initial system state distribution, and performance matrices ($Q$ and $R$). These two examples also use the same random initial points; curves of the same color start from the same initial point. We used the same setup for Examples \ref{example:2} and \ref{example:4}.

The final cost marked in this figure represents the cost value of the $i_{\rm max}$-th iterate. The convergence speed of Problem \ref{pro.dLQR_initial_estimate} (Examples \ref{example:3} and \ref{example:4}) is significantly faster than that of Problem \ref{pro.dLQR} (Examples \ref{example:1} and \ref{example:2}). In particular, all runs of Problem \ref{pro.dLQR_initial_estimate} converge within 100 iterations. The reason might be that for Problem \ref{pro.dLQR_initial_estimate}, all similarity transformations of $\mK^\ddagger$ in \eqref{eq.dlqr_2_cost} are globally optimal points, the gradient descent method can quickly converge to a certain globally optimal point that is closer to the initial controller. 

Instead, for Problem \ref{pro.dLQR} (see Fig. \ref{fig:curve_p1} and \ref{fig:curve_p2}), the gradient descent method may not converge to a fixed controller point within 20000 iterations. For example, the green and black curves of Fig. \ref{fig:curve_p1} and \ref{fig:curve_p2} do not converge to the optimal point, whose final iterate has a nonzero gradient since the limiting performance of these runs occurs when $B_{\mK}\rightarrow -\infty$ and $C_{\mK}\rightarrow 0$. This also demonstrates that the observable stationary point of Problem \ref{pro.dLQR} is unique. 

From Fig. \ref{fig:curve_p1} and \ref{fig:curve_p3}, we can observe that even the non-convergent curves of Example \ref{example:1} learn a dynamic controller that performs better than the optimal solution of Example \ref{example:3}. In particular, the final cost of green and black curves in Fig. \ref{fig:curve_p1} is about 25.4\% less than the optimal cost of Example \ref{example:3}. Similar results also hold between Examples \ref{example:2} and \ref{example:4}. This supports that \textit{designing an initial controller state $\xi_0$ correlated with the initial system state $x_0$, such that the cross-correlation matrix $X_{12}=\mathbb{E}[x_0\xi_0^\tr]$ is nonsingular, will facilitate learning a good dynamic controller}. 

\subsection{Two-dimensional examples}
In this section, we further illustrate our theoretical analysis with two 2-dimensional examples. 
\begin{example}
\label{example:5}
Consider a 2-dimensional example with 
$$A=\begin{bmatrix}
 1&\frac{1}{20}\\
 0&1
\end{bmatrix},\;B=\begin{bmatrix}
0\\
\frac{1}{20}
\end{bmatrix},\; C=\begin{bmatrix}
1\\0 
\end{bmatrix}^\tr,\;Q=5\begin{bmatrix}
1&1\\1&1 
\end{bmatrix}$$ 
and $R=1$. To define \texttt{dLQR}~\eqref{eq.dLQR}, we choose
\begin{equation} 
\nonumber
X=\mathbb{E}_{\bar{x}_0\sim \bar{\mathcal{D}}} \; [\bar{x}_0\bar{x}_0^\tr] = \begin{bmatrix}
0.2&0&0.05&0\\
0&0.8&0&0.05\\
0.05&0&0.2&0\\
0&0.05&0&0.8\\
\end{bmatrix}.
\end{equation}
By Theorem \ref{theorem.solution_expression}, the unique observable stationary point can be identified as
\begin{equation} 
\nonumber
{\mK}^\star = \begin{bmatrix}
0&-0.518&-0.183\\
4.392&-0.098&0.013\\
31.329&-8.247&0.854
\end{bmatrix}.
\end{equation}
\end{example}

\begin{example}
\label{example:6}
Consider the system in Example \ref{example:5}.  To define \texttt{dLQR}~\eqref{eq.dLQR}, we choose  
\begin{equation}
\nonumber
X=\mathbb{E}_{\bar{x}_0\sim \bar{\mathcal{D}}} \; [\bar{x}_0\bar{x}_0^\tr] = 
\begin{bmatrix}
\begin{matrix} 
0.2 & 0 \\ 0 & 0.8 \end{matrix} & \quad \text{\rm \large0}_{2\times 2} &\quad \\ \text{\rm \large0}_{2\times 2} & \quad \bigestimate &\quad
\end{bmatrix}
\end{equation}
with $V=1$. Theorem~\ref{theorem.optimal_equivalence} implies that all minimal stationary points identical up to a similarity transformation are globally optimal, which is 
\begin{equation} 
\nonumber
{\mK}^\star =\mathscr{T}_T({\mK}^\ddagger)
\end{equation}
with 
\begin{equation}
\nonumber
{\mK}^\ddagger=\begin{bmatrix}
0&-2.073& -2.920\\
0.448&0.552&0.050 \\
0.685&-0.788&  0.854
\end{bmatrix}.
\end{equation}
\end{example}

\begin{figure}[t]
\centering
\captionsetup{singlelinecheck = false,labelsep=period, font=small}
\captionsetup[subfigure]{justification=centering}
\subfloat[]{\label{fig:curve_p5}\includegraphics[width=0.5\linewidth]{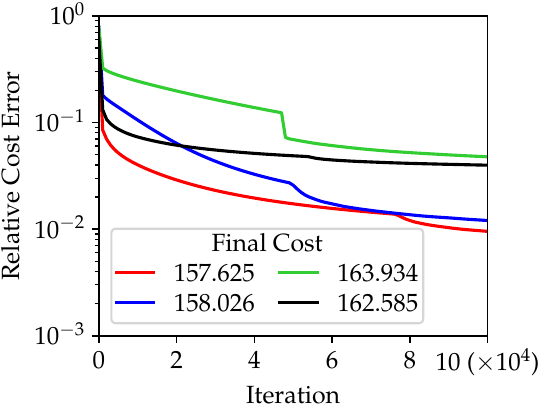}}
\subfloat[]{\label{fig:curve_p6}\includegraphics[width=0.5\linewidth]{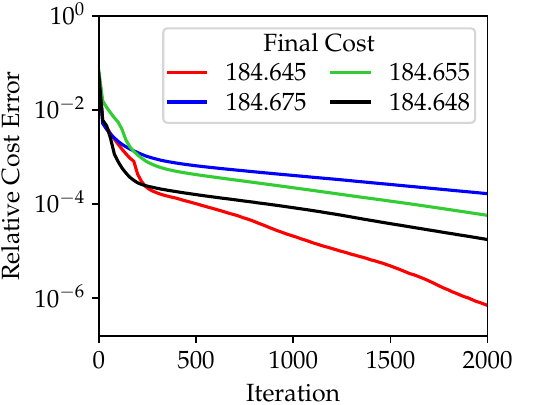}}
\caption{Learning curves of Examples \ref{example:5}-\ref{example:6} with four different random initialization. (a) Learning curves of Example \ref{example:5} with $J({\mK}^\star)=156.123$ and $i_{\max}=10^5$. (b) Learning curves of Example \ref{example:6} with $J({\mK}^\star)=184.645$ and $i_{\max}=2000$.}
\label{f:learning_curve_2}
\end{figure} 

Similar to the numerical results obtained for the one-dimensional examples, \texttt{dLQR} under the setting of Problem \ref{pro.dLQR_initial_estimate} converges faster than Problem \ref{pro.dLQR}; see Fig. \ref{f:learning_curve_2}. Actually, the learning curves of Example \ref{example:5} did not converge within $10^5$ iterations, although their final cost values were even smaller than the optimal cost of Example \ref{example:6}.  Overall, our numerical results indicate that Problem \ref{pro.dLQR_initial_estimate} enjoys faster convergence due to the global optimality of minimal stationary points, and that Problem \ref{pro.dLQR} may yield a better-limiting performance if we can design reasonable initial controller states. Further theoretical analysis is left for future work. 

 \balance
\section{Conclusion}
\label{sec:conclusion}
In this paper, we have analyzed the optimization landscape of linear quadratic control problems that use dynamic output-feedback policies. We have shown that the \texttt{dLQR} cost varies with similarity transformations, and identified the structure of the optimal similarity transformation of an observable stabilizing controller. More importantly, we characterized the stationary point of the policy gradient optimization and proved that the associated dynamic controller is unique if it is observable. This result provides a certificate of optimality for the solution of policy gradient methods.  We proved that the optimal solution of \texttt{dLQR} is equivalent to LQG when the initial controller state satisfies a certain structural constraint. In this case, all minimal stationary points are globally optimal, and they are identical up to a similarity transformation. 

Our work brings new insights for understanding the policy gradient algorithms for solving partially observed control or decision-making problems. Future work includes designing model-free dynamic controller learning methods, establishing convergence conditions for policy gradient algorithms, and investigating the global optimality and existence conditions of the observable stationary point in general cases. 

\appendix

In this appendix, we provide some auxiliary proofs and additional derivations for some results in the main text. 

\subsection{Proof of Lemma \ref{lemma.positive_P}}
\label{sec:proof_positive}
\begin{proof}
It is easy to show that 
$$\bar{Q} + \bar{C}^\tr \mK^\tr \bar{R}\mK\bar{C}=D^\tr D,$$
where $D:={\rm diag}
 (Q^{\frac{1}{2}}, R^{\frac{1}{2}}C_{\mK})$. 
By Lemma \ref{lemma.Lyapunov_stability}(c), $P_{\mK} \in  \mathbb{S}_{++}^{2n}$ if $(D,\bar{A}+\bar{B}\mK\bar{C})$ is observable. 
This is equivalent to the eigenvalues of
\begin{equation}
\label{eq.matrix_observability}
\begin{bmatrix}
 A+Z_{11}Q^{\frac{1}{2}}    & BC_{\mK} + Z_{12}R^{\
 \frac{1}{2}}C_{\mK} \\
B_{\mK}C + Z_{21}Q^{\frac{1}{2}}    & A_{\mK}+Z_{22}R^{\
 \frac{1}{2}}C_{\mK}
\end{bmatrix}
\end{equation}
should be freely assigned by choosing $Z_{11}$, $Z_{12}$, $Z_{21}$, and $Z_{22}$.  

Let $Z_{12}=-BR^{-\frac{1}{2}}$, we can easily show that the eigenvalues of \eqref{eq.matrix_observability} can be arbitrarily assigned if $(Q^{\frac{1}{2}},A)$ and $(C_{\mK},A_{\mK})$ are both observable.  
Thus, it completes the proof.
\end{proof}

\subsection{Non-existence of the optimal similarity transformation}
\label{appendix.non-existence}
We take a one-dimensional system as an example (i.e., $x_t$ and $\xi_t$ are scalars), to show the non-existence of the optimal similarity transformation if $X_{12}$ is singular. 
Under similarity transformation~\eqref{eq:similarity-transformation}, we have 
\begin{equation}
\label{eq.u_sequence}
\begin{aligned}
u_0 &= C_{\mK}T^{-1}\xi_0, \\
u_1 &= C_{\mK}T^{-1}( TA_{\mK}T^{-1}\xi_0 + TB_{\mK}y_0) \\
& = C_{\mK}A_{\mK}T^{-1}\xi_0 + C_{\mK}B_{\mK}y_0.
\end{aligned}
\end{equation}
Given an observable stabilizing controller $\mK$, by \eqref{eq.optimal_tranforsmation} of  \Cref{proposition.optimal_tansformation}, one has 
\begin{equation}
\label{eq.T_limit}
\lim_{X_{12}\rightarrow 0} (T^\star)^{-1} = \lim_{X_{12}\rightarrow 0} - P_{22}^{-1}P_{12}^\tr X_{12}X_{22}^{-1} =  0.
\end{equation}
Note that the cross-correlation value $X_{12}=0$ if the initial controller state $\xi_0$ is zero-mean and independent of the initial system state $x_0$. Using \eqref{eq.T_limit} in \eqref{eq.u_sequence}, we can observe that the controller input $u_t$ tends to ignore the influence of $\xi_0$ by increasing $T$ in this one-dimensional instance. This is because the initial controller state provides no information for the estimation of the initial system state if $X_{12}=0$.

\subsection{Proof of Proposition \ref{proposition.solution_of_riccati}}
\label{sec.proof_of_riccati}
\begin{proof}
Define $\phi_{L_i}:=A-L_iC$.  Let $\hat{\Sigma}_{L_i}$, $i=0,1,2,\cdots$, be the solutions of the equation
\begin{equation}
\label{eq.lyapunov_sigma_mini}
\hat{\Sigma}_{L_i}=\Delta_X  + \phi_{L_i}\hat{\Sigma}_{L_i}\phi_{L_i}^\tr ,
\end{equation}
where 
\begin{equation}
\label{eq.PIM}
L_{i+1} = A \hat{\Sigma}_{L_i} C^\tr (C\hat{\Sigma}_{L_i} C^\tr)^{-1},
\end{equation}
$L_0 \in \mathbb{L}$, and $\Delta_X$ is defined in \eqref{eq.definition_delta_X}. Next, we will show that \eqref{eq.PIM} is well-defined and 
\begin{equation}
\label{eq.nonincreasing_sigma}
\hat{\Sigma}_{L_0} \succeq \hat{\Sigma}_{L_1} \succeq \cdots \succeq \hat{\Sigma}_{L_\infty} \succ 0.
\end{equation}
Note that \eqref{eq.nonincreasing_sigma} leads to a monotonically non-increasing sequence ${\rm Tr}(\hat{\Sigma}_{L_i})$.

Since $X\succ 0$, by the Schur complement, one has $\Delta_X \succ 0$.  Since $\rho(\phi_{L_0})<1$ and by Lemma \ref{lemma.Lyapunov_stability}(b), the unique positive definite solution $\hat{\Sigma}_{L_0}$ of \eqref{eq.lyapunov_sigma_mini} can be written as
\begin{equation}
\nonumber
\hat{\Sigma}_{L_0}=\sum_{j=0}^{\infty}\phi_{L_0}^j\Delta_X(\phi_{L_0}^\tr)^j.
\end{equation}

Letting $L_1$ in \eqref{eq.PIM}, we observe the following identity
\begin{equation}
\label{eq.sigma_1_and_0}
 \begin{aligned}
 \phi_{L_0}\hat{\Sigma}_{L_0}\phi_{L_0}^\tr&=\phi_{L_1}\hat{\Sigma}_{L_0}\phi_{L_1}^\tr\\
 &\quad+ (L_1-L_0)C\hat{\Sigma}_{L_0} C^\tr(L_1-L_0)^\tr.
 \end{aligned}  
\end{equation}
Plugging \eqref{eq.sigma_1_and_0} in \eqref{eq.lyapunov_sigma_mini}, $\hat{\Sigma}_{L_0}$ also satisfies the equation 
\begin{equation}
\nonumber
 \hat{\Sigma}_{L_0}=\phi_{L_1}\hat{\Sigma}_{L_0}\phi_{L_1}^\tr + M=\sum_{j=0}^{\infty}\phi_{L_1}^jM(\phi_{L_1}^\tr)^j,
\end{equation}
where 
$ M = \Delta_X + (L_1-L_0)C\hat{\Sigma}_{L_0} C^\tr(L_1-L_0)^\tr \succ 0.$ 
This implies that $\rho(\phi_{L_1}) <1$ and that $\hat{\Sigma}_{L_1} \in \mathbb{S}_{++}^n$ by Lemma \ref{lemma.Lyapunov_stability}(b).

Using \eqref{eq.sigma_1_and_0} with $\hat{\Sigma}_{L_0}$ and $\hat{\Sigma}_{L_1}$ given by \eqref{eq.lyapunov_sigma_mini}, we obtain
\begin{equation}
\nonumber
 \begin{aligned}
\hat{\Sigma}_{L_0} -\hat{\Sigma}_{L_1}
 & =\phi_{L_1}(\hat{\Sigma}_{L_0}-\hat{\Sigma}_{L_1})\phi_{L_1}^\tr \\
 &\qquad + (L_1-L_0)C\hat{\Sigma}_{L_0} C^\tr (L_1-L_0)^\tr\\
 &=\sum_{j=0}^{\infty}\phi_{L_1}^j(L_1-L_0)C\hat{\Sigma}_{L_0} C^\tr(L_1-L_0)^\tr(\phi_{L_1}^\tr)^j\\
 &\succeq 0 .
 \end{aligned}  
\end{equation}
We can easily obtain \eqref{eq.nonincreasing_sigma}
and 
$\rho(\phi_{L_i}) <1$ for $\forall i \in \mathbb{N}$ by induction. 

According to \eqref{eq.lyapunov_sigma_mini}, $\hat{\Sigma}_L \succeq \Delta_X \succ 0$ for all $L \in \mathbb{L}$. This means $\hat{\Sigma}_{L_i}$ must be bounded below by a certain positive definite matrix. Combining this with \eqref{eq.nonincreasing_sigma}, we can define \begin{equation}
\nonumber
{\hat{\Sigma}}:=\lim_{i \rightarrow \infty }\hat{\Sigma}_{L_i},
\end{equation}
and such that 
\begin{equation}
\nonumber
L^\star:=\lim_{i \rightarrow \infty }L_i = A \hat{\Sigma} C^\tr (C\hat{\Sigma} C^\tr)^{-1}\in \mathbb{L}.
\end{equation}
By plugging ${\hat{\Sigma}}$ and $L^\star$ in \eqref{eq.lyapunov_sigma_mini}, we observe that $\hat{\Sigma}\in \mathbb{S}_{++}^n$ is the positive definite solution to \eqref{eq.sigma_riccati}.

Similar to \eqref{eq.sigma_1_and_0}, for $\forall L \in \mathbb{L}$, we can further derive that 
\begin{equation}
\label{eq.sigma_L_and_optimal}
 \begin{aligned}
\phi_L\hat{\Sigma}\phi_L^\tr= \phi_{L^\star}\hat{\Sigma}\phi_{L^\star}^\tr+ (L^\star-L)C\hat{\Sigma} C^\tr(L^\star-L)^\tr.
 \end{aligned}  
\end{equation}
Using \eqref{eq.sigma_L_and_optimal} and \eqref{eq.lyapunov_sigma_mini}, we obtain
\begin{equation}
\label{eq.optimal_sigma}
 \begin{aligned}
\hat{\Sigma}_L &-\hat{\Sigma}\\
 & =\phi_L \hat{\Sigma}_L\phi_L^\tr-\phi_{L^\star} \hat{\Sigma}\phi_{L^\star}^\tr\\
  & =\phi_L \hat{\Sigma}_L\phi_L^\tr-\phi_{L} \hat{\Sigma}\phi_{L}^\tr+\phi_{L} \hat{\Sigma}\phi_{L}^\tr-\phi_{L^\star} \hat{\Sigma}\phi_{L^\star}^\tr\\
  & =\phi_L(\hat{\Sigma}_L- \hat{\Sigma})\phi_L^\tr+ (L^\star-L)C\hat{\Sigma} C^\tr(L^\star-L)^\tr\\
  &=\sum_{j=0}^{\infty}\phi_{L}^j(L^\star-L)C\hat{\Sigma} C^\tr(L^\star-L)^\tr(\phi_{L}^\tr)^j\\
 &\succ 0, \quad \forall L\in \mathbb{L}\backslash L^\star.
 \end{aligned}  
\end{equation}
So far, we have proved that the positive definite solution to \eqref{eq.sigma_riccati} exists, and $L^\star$ in \eqref{eq.optimal_L} is the optimal solution to \eqref{eq.problem_of_observer}.

As for the uniqueness, let $\hat{\Sigma}' \in \mathbb{S}_{++}^n$ be another solution to \eqref{eq.sigma_riccati}. Similar to \eqref{eq.optimal_sigma}, we have 
\begin{equation}
\nonumber
\hat{\Sigma}_L-\hat{\Sigma}'\succ 0, \quad \forall L\in \mathbb{L}\backslash \{L'\},
\end{equation}
where $L':=A \hat{\Sigma}' C^\tr (C\hat{\Sigma}' C^\tr)^{-1}$.
This is contrary to that $L^\star$ is the globally optimal solution of problem \eqref{eq.problem_of_observer}, which establishes uniqueness. 
\end{proof}

\subsection{Some detailed derivations}
\subsubsection{Derivation of \eqref{eq.K_original}}
\label{sec.detail_derivation_of_K_original}

Let \eqref{eq.gradient} equal to 0, we have
\begin{subequations}
\begin{align}
&\begin{aligned}
\label{eq.K_12_stationary}
C_{\mK^\star}&=-(R + B^\top P_{11} B)^{-1}(B^\top P_{12}A_{\mK^\star} \\
&\qquad +B^\top  (P_{11}A+P_{12}B_{\mK^\star}C)\Sigma_{12}\Sigma_{22}^{-1}),
\end{aligned}\\
&\begin{aligned}
\label{eq.K_21_stationary}
B_{\mK^\star} &= -(P_{22}^{-1} P_{12}^\top A\Sigma_{11} + P_{22}^{-1}P_{12}^\top B C_{\mK^\star}\Sigma_{12}^\top\\
&\qquad +A_{\mK^\star}\Sigma_{12}^\top) C^\top (C\Sigma_{11} C^\top)^{-1},
\end{aligned}\\
&\begin{aligned}
\label{eq.K_22_stationary}
A_{\mK^\star}&=-P_{22}^{-1}P_{12}^\top B C_{\mK^\star} - P_{22}^{-1}P_{12}^\top A\Sigma_{12}\Sigma_{22}^{-1}\\
&\qquad-B_{\mK^\star}C\Sigma_{12}\Sigma_{22}^{-1}.
\end{aligned}
\end{align}
\end{subequations}

By plugging \eqref{eq.K_22_stationary} in \eqref{eq.K_12_stationary}, we get \eqref{eq.K_12_original}. Similarly, by plugging \eqref{eq.K_22_stationary} in \eqref{eq.K_21_stationary}, one has \eqref{eq.K_21_original}. Finally, if we use both \eqref{eq.K_12_original} and \eqref{eq.K_21_original} in \eqref{eq.K_22_stationary}, we arrive at \eqref{eq.K_22_original}.

\subsubsection{Derivation of \eqref{eq.inverse}}
\label{sec.detail_derivation_of_inverse}
From \eqref{eq.block_lyapunov_P12} and \eqref{eq.block_lyapunov_P22}, we observe that 
\begin{equation}
\nonumber
\begin{aligned}
&P_{12}^\tr \Sigma_{12} + P_{22} \Sigma_{22}\\
\quad = &(C_{\mK}^\tr B^\tr P_{11} A + A_{\mK}^\tr P_{12}^\tr A + C_{\mK}^\tr B^\tr P_{12} B_{\mK}C \\
&\qquad \qquad \qquad \qquad \qquad \qquad  \ +A_{\mK}^\tr P_{22}B_{\mK}C)\Sigma_{12} \\
&\qquad +(C_{\mK}^\tr R C_{\mK}+ C_{\mK}^\tr B^\tr P_{11} B C_{\mK} + A_{\mK}^\tr P_{12}^\tr B C_{\mK} \\
&\qquad\qquad\qquad+ C_{\mK}^\tr B^\tr P_{12} A_{\mK}+A_{\mK}^\tr P_{22}A_{\mK})\Sigma_{22}.
\end{aligned}
\end{equation}
Plugging \eqref{eq.K_12_original}, \eqref{eq.K_21_original}, and \eqref{eq.K_22_original} in the above equation, we can easily derive that 
\begin{equation}
\nonumber
\begin{aligned}
&P_{12}^\tr \Sigma_{12} + P_{22} \Sigma_{22}\\
\quad = &(C_{\mK}^\tr B^\tr P_{11} A + A_{\mK}^\tr P_{12}^\tr A + C_{\mK}^\tr B^\tr P_{12} B_{\mK}C \\
&\qquad \qquad \qquad \qquad \qquad \qquad  \ +A_{\mK}^\tr P_{22}B_{\mK}C)\Sigma_{12} \\
&\qquad -(C_{\mK}^\tr R K^\star+ C_{\mK}^\tr B^\tr P_{11} B K^\star + A_{\mK}^\tr P_{12}^\tr B K^\star \\
&\qquad\qquad\qquad +C_{\mK}^\tr B^\tr P_{12}P_{22}^{-1}P_{12}^\tr (A-BK^\star-L^\star C)\\
& \qquad\qquad\qquad +A_{\mK}^\tr P_{12}^\tr (A-BK^\star-L^\star C))\Sigma_{12}\\
\quad =&C_{\mK}^\tr ( B^\tr \widetilde{P} A  - (R+B^\tr \widetilde{P}B) K^\star)\Sigma_{12}\\
\quad \overset{1}{=}&C_{\mK}^\tr ( B^\tr \widetilde{P} A - (R+B^\tr \widetilde{P}B) (R+B^\tr {\widetilde{P}}B)^{-1}B^\tr {\widetilde{P}}A)\Sigma_{12}\\
\quad = &0,
\end{aligned}
\end{equation}
where step 1 follows by \eqref{eq.optimal_K}.

\subsubsection{Derivation of \eqref{eq.relation_P_and_x}}
\label{sec.detail_derivation_of_T}
From \eqref{eq.block_lyapunov_sigma12} and \eqref{eq.block_lyapunov_sigma22}, \eqref{eq.inverse} can be rewritten as 
\begin{equation}
\label{eq.relation_P_and_x_expand}
\begin{aligned}
&P_{12}^\tr \Sigma_{12} + P_{22} \Sigma_{22}\\
\quad = & P_{12}^\tr (X_{12}+  A \Sigma_{11}C^\tr B_{\mK}^\tr + B C_{\mK} \Sigma_{12}^\tr C^\tr B_{\mK}^\tr \\
&\qquad \qquad + A \Sigma_{12} A_{\mK}^\tr +B C_{\mK} \Sigma_{22} A_{\mK}^\tr)
\\
& + P_{22}(X_{22}+B_{\mK} C \Sigma_{11} C^\tr B_{\mK}^\tr + A_{\mK}\Sigma_{12}^\tr C^\tr B_{\mK}^\tr \\
&\qquad \qquad + B_{\mK}C \Sigma_{12} A_{\mK}^\tr+A_{\mK} \Sigma_{22}A_{\mK}^\tr)\\
\quad=&0.
\end{aligned}
\end{equation}
Plugging \eqref{eq.K_12_original}, \eqref{eq.K_21_original}, and \eqref{eq.K_22_original} in \eqref{eq.relation_P_and_x_expand}, one has
\begin{equation}
\nonumber
\begin{aligned}
&P_{12}^\tr \Sigma_{12} + P_{22} \Sigma_{22}\\
\quad = &P_{12}^\tr X_{12}  + P_{22}X_{22} \\
&\qquad +  P_{12}^\tr(A \Sigma_{11}C^\tr B_{\mK}^\tr + B C_{\mK} \Sigma_{12}^\tr C^\tr B_{\mK}^\tr \\
&\qquad \qquad + A \Sigma_{12} A_{\mK}^\tr +B C_{\mK} \Sigma_{22} A_{\mK}^\tr)
\\
&\qquad -P_{12}^\tr(L^\star C \Sigma_{11} C^\tr B_{\mK}^\tr + L^\star C \Sigma_{12} A_{\mK}^\tr \\
& \qquad \qquad + (A-BK^\star-L^\star C)\Sigma_{12}\Sigma_{22}^{-1}\Sigma_{12}^\tr C^\tr B_{\mK}^\tr \\
&\qquad \qquad +(A-BK^\star-L^\star C) \Sigma_{12}A_{\mK}^\tr)\\
\quad = & P_{12}^\tr X_{12}  + P_{22}X_{22} +  P_{12}^\tr (A-L^\star C) \hat{\Sigma}C^\tr B_{\mK}^\tr\\
\quad \overset{1}{=} &P_{12}^\tr X_{12}  + P_{22}X_{22} \\
&\qquad +  P_{12}^\tr (A-A{\hat{\Sigma}}C^\tr (C{\widetilde{\Sigma}}C^\tr)^{-1}C) \widetilde{\Sigma}C^\tr B_{\mK}^\tr\\
\quad = &P_{12}^\tr X_{12}  + P_{22}X_{22} \\
\quad=&0,
\end{aligned}
\end{equation}
where step 1 follows by  \eqref{eq.optimal_L}.

\subsection{Further analysis on the equivalence between \texttt{dLQR} and LQG}
\label{appendix.general_equi}
In Problem \ref{pro.dLQR_initial_estimate}, the initial controller state is assumed to be zero-mean and independent of the initial system state. We here discuss a more general setting. 
\begin{problem}[A general version of Problem \ref{pro.dLQR_initial_estimate}]
\label{pro.dLQR_initial_estimate_general}
\begin{equation}  
\nonumber
\begin{aligned}
\min_{\mK} \quad&  J(\mK)\\
\text{\rm subject to} \quad &\mK\in \mathbb{K},
\end{aligned}
\end{equation}
where $J(\mK)$ is defined in~\eqref{eq.cost_in_P} and $\mathbb{K}$ is given in~\eqref{eq:stabilizing-K}. The initial controller state $\xi_0$ here satisfies the structural constraint \eqref{eq.structure_constraint}, where $s\in \mathbb{R}^d$ is randomly sampled from the distribution $\mathcal{D}_s$, which can be correlated with $x_0$ with $\mathbb{E}[x_0s^\tr]=M$. In this case, 
\begin{equation}
\label{eq.general_initial_matrix}
X= \mathbb{E}_{\bar{x}_0\sim \bar{\mathcal{D}}}[\bar{x}_0\bar{x}_0^\tr]= \begin{bmatrix}
   X_{11}  & MB_\mK^\tr \\
    B_\mK M^\tr & B_{\mK} V B_{\mK}^\tr
\end{bmatrix},
\end{equation}
where $V = \mathbb{E}_{s\sim \mathcal{D}_s}[ss^\tr]$ is fixed. 
\end{problem}

\begin{proposition}
\label{proposition.equivalence}
Suppose $C$ has full row rank, $V \in \mathbb{S}_{++}^d$, and Assumption \ref{assumption.control_observe} holds. All stationary points $\mK^\star \in  \mathbb{K}\cap\mathbb{K}_o \cap \mathbb{K}_s$ to Problem \ref{pro.dLQR_initial_estimate_general} are globally optimal if $(\bar{A}+\bar{B}\mK^\star\bar{C},X^{\frac{1}{2}})$ is reachable, and they are in the form of \eqref{eq.optimal_K_LQG} with $\mK^{\ddagger}$ defined in \eqref{eq.dlqr_2_cost} and $K^\star$ defined in \eqref{eq.optimal_K}. Compared with Theorem \ref{theorem.optimal_equivalence}, the only difference is that 
\begin{equation}
\nonumber
L^\star = (A{\hat{\Sigma}} C^\tr + M) (C{\hat{\Sigma}} C^\tr+V)^{-1},
\end{equation}
where ${\hat{\Sigma}}$ being the unique positive definite solution to the following Riccati equation 
\begin{equation}
\nonumber
\begin{aligned}
&{\hat{\Sigma}}=X_{11}+A {\hat{\Sigma}} A^\tr \\
&\qquad - (A {\hat{\Sigma}}C^\tr +M) (C{\hat{\Sigma}}\bar{C}^\tr+ V)^{-1} (A {\hat{\Sigma}}C^\tr +M)^\tr.
\end{aligned}
\end{equation}
\end{proposition}
\begin{proof}
The proof is very similar to that of Theorems \ref{theorem.solution_expression} and \ref{theorem.optimal_equivalence}. Some key points of this proof are as follows. The gradient of the \texttt{dLQR} cost w.r.t. $B_{\mK}$ is
\begin{equation}
\nonumber
\begin{aligned}
&\nabla_{B_{\mK}} J(\mK)=2P_{12}^\tr (A \Sigma_{11} C^\tr +M)+\\
&\ 2\left(P_{22}B_{\mK}(C\Sigma_{11} C^\tr +V)+(P_{12}^\tr B C_{\mK}+P_{22}A_{\mK})\Sigma_{12}^\tr C^\tr\right).
\end{aligned}
\end{equation}
And by Lemma \ref{lemma.Lyapunov_stability}(c), the positive definiteness of $\Sigma_{\mK^\star}$ can be guaranteed by the reachability of $(\bar{A}+\bar{B}\mK^\star\bar{C},X^{\frac{1}{2}})$.
\end{proof}

As shown in Proposition \ref{proposition.equivalence}, the optimal solution of Problem \ref{pro.dLQR_initial_estimate_general} is identical to that of the discrete-time LQG in Problem \ref{pro.LQG} when the process noise and measurement noise are correlated with $X_{11} = \mathbb{E}[w_tw_t^\tr],\; V = \mathbb{E}[v_tv_t^\tr], \;M = \mathbb{E}[w_tv_t^\tr]$, and $L^\star$ is the associated Kalman gain \cite{kwong1991linear}.


\ifCLASSOPTIONcaptionsoff
  \newpage
\fi
\bibliographystyle{ieeetr}
\bibliography{ref}

\newpage
\vskip -2\baselineskip plus -1fil
\begin{IEEEbiography}[{\includegraphics[width=1in,height=1.25in,clip,keepaspectratio]{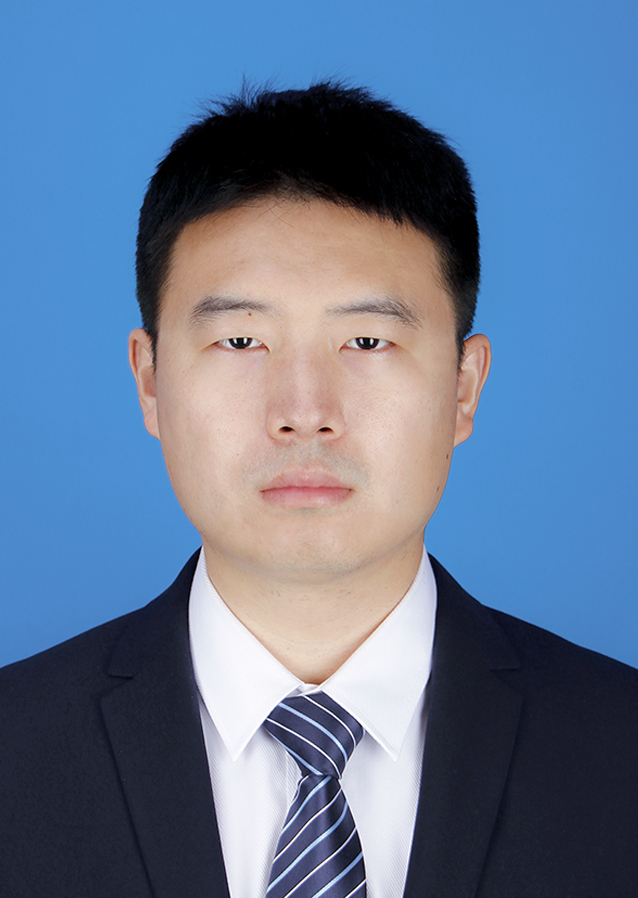}}]{Jingliang Duan} received his Ph.D. degree in the School of Vehicle and Mobility, Tsinghua University, China, in 2021.  He studied as a visiting student researcher in the Department of Mechanical Engineering, University of California, Berkeley, in 2019, and worked as a research fellow in the Department of Electrical and Computer Engineering, National University of Singapore, in from 2021 to 2022. He is currently an associate professor in the School of Mechanical Engineering, University of Science and Technology Beijing, China. His research interests include reinforcement learning, optimal control, and self-driving decision-making.
\end{IEEEbiography}

\vskip -2\baselineskip plus -1fil
\begin{IEEEbiography}[{\includegraphics[width=1in,height=1.25in,clip,keepaspectratio]{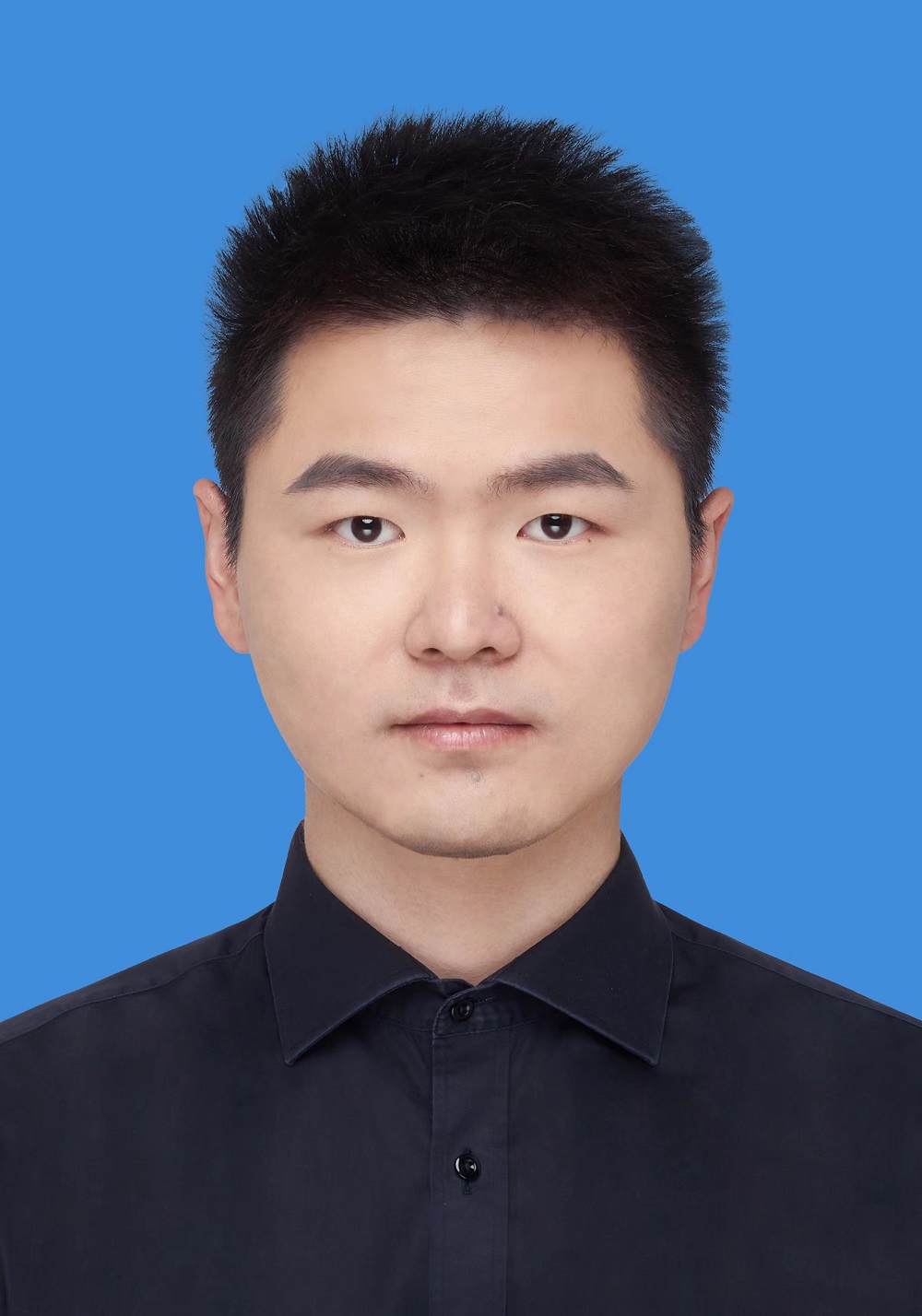}}]{Wenhan Cao} received his B.E. degree in the School of Electrical Engineering from Beijing Jiaotong University, Beijing, China, in 2019. He is currently a Ph.D. candidate in the School of Vehicle and Mobility, Tsinghua University, Beijing, China. His research interests include optimal filtering and reinforcement learning. He was a finalist for the Best Student Paper Award at the 2021 IFAC MECC.
\end{IEEEbiography}

\vskip -2\baselineskip plus -1fil
\begin{IEEEbiography}[{\includegraphics[width=1in,height=1.25in,clip,keepaspectratio]{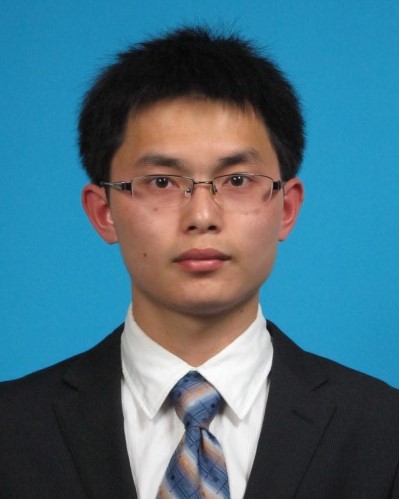}}]{Yang Zheng} received the B.E. and M.S. degrees from Tsinghua University, Beijing, China, in 2013 and 2015, respectively, and the D.Phil. (Ph.D.) degree in Engineering Science from the University of Oxford, U.K., in 2019. He is currently an assistant professor with the Department of Electrical and Computer Engineering, UC San Diego. He was a research associate at Imperial College London and was a postdoctoral scholar in SEAS and CGBC at Harvard University. His research interests focus on learning, optimization, and control of network systems, and their applications to autonomous vehicles and traffic systems.

Dr. Zheng was a finalist (co-author) for the Best Student Paper Award at the 2019 ECC. He received the Best Student Paper Award at the 17th IEEE ITSC in 2014, the Best Paper Award at the 14th Intelligent Transportation Systems Asia-Pacific Forum in 2015, and the 2022 Best Paper Award in the IEEE Transactions on Control of Network Systems. He was a recipient of the National Scholarship, Outstanding Graduate in Tsinghua University, the Clarendon Scholarship at the University of Oxford, and the Chinese Government Award for Outstanding Self-financed Students Abroad. Dr. Zheng won the 2019 European Ph.D. Award on Control for Complex and Heterogeneous Systems.
\end{IEEEbiography}

\vskip -2\baselineskip plus -1fil
\begin{IEEEbiography}[{\includegraphics[width=1in,height=1.25in,clip,keepaspectratio]{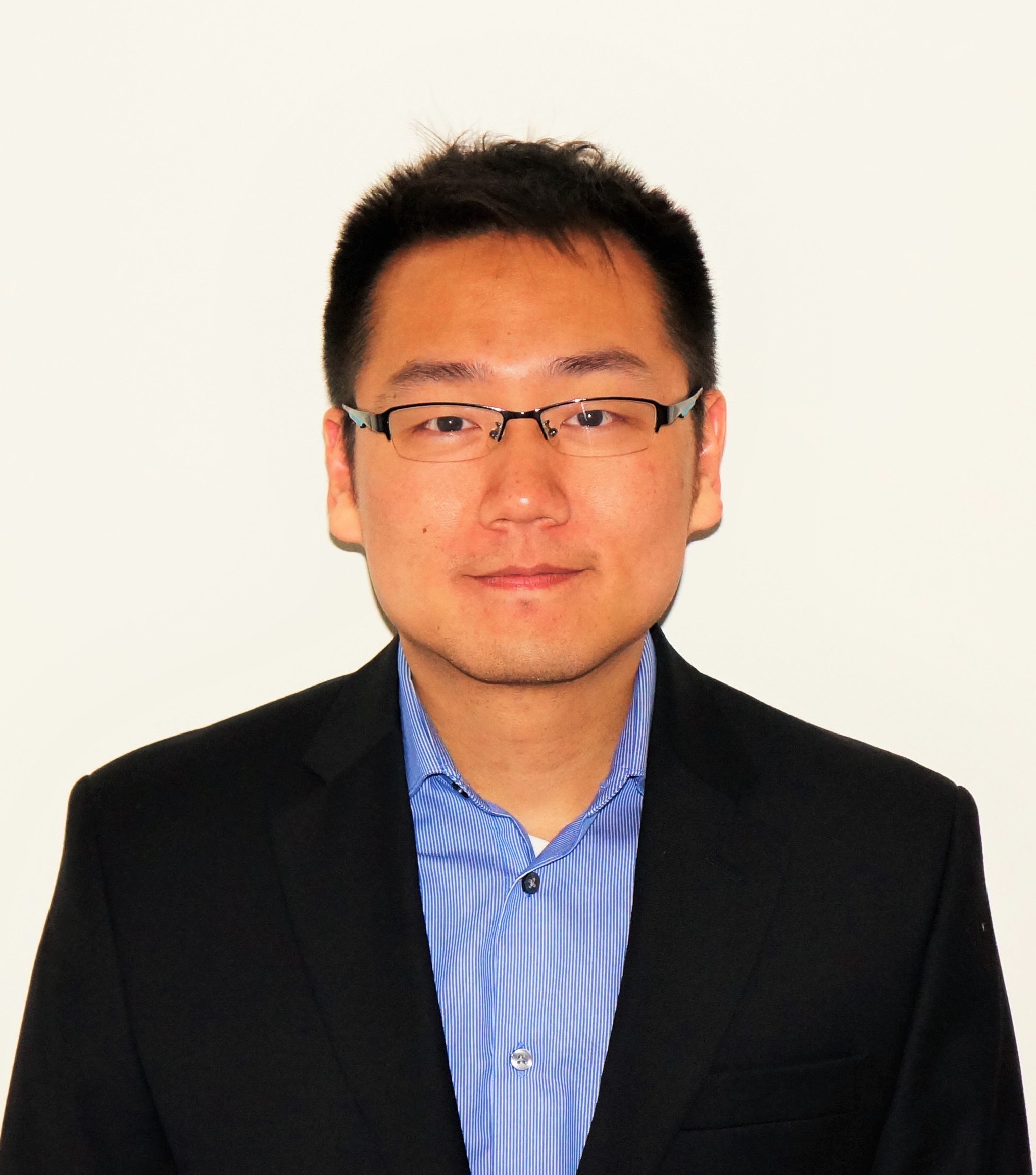}}]{Lin Zhao} received a B.S. and an M.S. degree in automatic control from the Harbin Institute of Technology, Harbin, China, in 2010 and 2012, respectively, an M.S. degree in mathematics and a Ph.D. degree in electrical and computer engineering from The Ohio State University, Columbus, OH, USA, in 2017. From 2018 to early 2020, He was a research scientist at Aptiv Pittsburgh Technology Center (now Motional), Pittsburgh, PA, USA. He is currently an Assistant Professor in the Department of Electrical and Computer Engineering, National University of Singapore, Singapore. His current research focuses on control and reinforcement learning with applications in robotics.
\end{IEEEbiography}

\end{document}